\documentclass[11pt]{amsart}
\usepackage{cite}
\usepackage{graphicx}
\usepackage{latexsym,amsmath,amsfonts,amscd, amsthm}
\usepackage{color}
\usepackage[colorlinks=true]{hyperref}
\hypersetup{urlcolor=blue, citecolor=red}
\usepackage{tikz}
\usetikzlibrary{arrows,backgrounds,snakes}
\usepackage{epsfig}
\usepackage{dsfont}
\usepackage{amssymb}
\usepackage{xcolor}
\usepackage{comment}
\usepackage{subfigure}
\usepackage{multirow}
\usepackage{mathrsfs}

\newtheoremstyle{plainNoItalics}{}{}{\normalfont}{}{\bfseries}{.}{ }{}
\theoremstyle{plain}
\newtheorem{thm}{Theorem}[section]

\newtheorem{defn}[thm]{Definition}

\newtheorem{prop}[thm]{Proposition}
\newtheorem{exa}[thm]{Example}

\newcommand{\beq}{\begin{equation}}
\newcommand{\eeq}{\end{equation}}
\newcommand{\beqa}{\begin{eqnarray}}
\newcommand{\eeqa}{\end{eqnarray}}
\newcommand{\bit}{\begin{itemize}}
\newcommand{\eit}{\end{itemize}}
\newcommand{\bedef}{\begin{defn}}
\newcommand{\edefn}{\end{defn}}
\newcommand{\bpro}{\begin{prop}}
\newcommand{\epro}{\end{prop}}

\newcommand{\mO}{{\mathcal O}}

\newcommand{\eps}{\varepsilon}

\newcommand{\bx}{{\bf x}}

\newcommand{\bu}{{\bf u}}
\newcommand{\bn}{{\bf n}}

\newcommand{\bV}{{\bf V}}

\newcommand{\sr}{{\text{Sr}}}
\newcommand{\fr}{{\text{Fr}}}

\title[High order AP schemes for SWEs in all-Fraude number]{High order well-balanced asymptotic preserving finite difference WENO schemes for the shallow water equations in all Froude numbers}


\begin{document}
	
\maketitle
\medskip
\centerline{\scshape Guanlan Huang}
\medskip
{\footnotesize
\centerline{School of Mathematical Sciences, Xiamen University}
\centerline{Xiamen, Fujian, 361005, P.R. China}
\centerline{Email: glhuang@stu.xmu.edu.cn}
}

\medskip
\centerline{\scshape Yulong Xing\footnote{The work of this author was partially supported by the NSF grant DMS-1753581.}}
\medskip
{\footnotesize
	\centerline{Department of Mathematics, The Ohio State University}		
	\centerline{Columbus, OH 43210, USA}
	\centerline{Email: xing.205@osu.edu}
	%\centerline{The work of this author was partially supported by the NSF grant DMS-1753581.}
}

\medskip
\centerline{\scshape Tao Xiong\footnote{Corresponding author. The work of this author was partially supported by NSFC grant No. 11971025, NSF grant of Fujian Province No. 2019J06002, and the Strategic Priority Research Program of Chinese Academy of Sciences Grant No. XDA25010401.}}
\medskip
{\footnotesize
   % please put the address of the author
\centerline{School of Mathematical Sciences, Xiamen University}
\centerline{Fujian Provincial Key Laboratory of Mathematical Modeling and High-Performance Scientific Computing}
\centerline{Xiamen, Fujian, 361005, P.R. China}
\centerline{Email: txiong@xmu.edu.cn}
}

\bigskip

\begin{abstract}
In this paper, high order semi-implicit well-balanced and asymptotic preserving finite difference WENO schemes are proposed for the shallow water equations with a non-flat bottom topography. We consider the Froude number ranging from $\mathcal{O}(1)$ to $0$, which in the zero Froude limit becomes the ``lake equations" for balanced flow
without gravity waves. We apply a well-balanced finite difference WENO reconstruction, coupled with a stiffly accurate implicit-explicit (IMEX) Runge-Kutta time discretization. The resulting semi-implicit scheme can be shown to be well-balanced, asymptotic preserving (AP) and asymptotically accurate (AA) at the same time. Both one- and two-dimensional numerical results are provided to demonstrate the high order accuracy, AP property and good performance of the proposed methods in capturing small perturbations of steady state solutions.
\end{abstract}

\vspace{0.1cm}

%\tableofcontents

\section{Introduction}
\label{sec1}
\setcounter{equation}{0}
\setcounter{figure}{0}
\setcounter{table}{0}

Shallow water equations (SWEs) are widely used in the modeling of water motion flows in rivers and coastal areas. They have important applications in ocean currents and hydraulic engineering, see, e.g. \cite{majda2003,klein2004applied,vallis2017atmospheric}. Considering the water flow in river, reservoir or open channels with a non-flat bottom, the SWEs can be written as follows:
\begin{equation}
\left\{
\begin{array}{ll}
  h_t + \nabla \cdot (h\bu)   = 0, \\ [3mm]
  (h\bu)_t + \nabla \cdot (h\bu\otimes \bu) + g\nabla(h^2/2) = -gh\nabla b,
\end{array}\right.
  \label{S1_E1}
\end{equation}
where $h$ is the depth of the water layer, $\bu$ is the flow velocity, defined on a time-space domain $(t,\bx)\in\mathbb{R}^+\times\Omega$. $g$ is the gravitational constant and $b(\bx)$ is the bottom topography which is independent of time. $\otimes$ denotes the Kronecker product. When the bottom is flat, this system is equivalent to the isentropic Euler system in the homogeneous case. However, the geometrical source term changes the property of the system when a non-flat bottom topography is taken into account.

Many shock capturing schemes with explicit time discretizations have been developed to solve the SWEs with source term \eqref{S1_E1}, including high order finite difference \cite{vukovic2002eno,xing2005high,gao2017high,li2020high,wang2020new}, finite volume \cite{audusse2004fast,noelle2006well,CGP2006,KP2007,noelle2007high,xing2011advantage,liu2021new}, residual distribution methods \cite{ricchiuto2009stabilized,Ricchiuto2015} and discontinuous Galerkin schemes \cite{XZS2010,xing2014exactly,wen2020entropy,zhang2021high}, and many references therein. When solving the SWEs with source term numerically, it is important to preserve the exact conservation property (C-property) \cite{leveque1998balancing}, namely, the nonzero flux gradient should be exactly balanced by the source term in the case of a stationary water. Such schemes are named well-balanced methods. During the past few decades, there have been extensive studies on the design and analysis of well-balanced methods for various hyperbolic equations with source terms. For the SWEs \eqref{S1_E1}, the still-water stationary solution takes the form
\begin{equation}
\label{still}
h\bu = 0, \quad h+b=\text{Const.}
\end{equation}
%for a moving water, it becomes
%\begin{equation}
%h\bu = \text{Const.}, \quad \frac12|\bu|^2+g(h+b)=\text{Const.}.
%\end{equation}
Well-balanced schemes for the SWEs are able to capture small perturbations of the hydrostatic or nearly hydrostatic flows on a coarse mesh, and we refer to the above list of literatures and the survey papers \cite{Xing14,Kurganov18} for more discussions.

On the other hand, if we choose a characteristic length $l_0$, a characteristic depth $h_0$, a characteristic velocity $U_0$ and a characteristic time $t_0$, we can define the following dimensionless variables
  \begin{equation}
  \label{S1_E2}
  \hat{{\bx}} = \frac{{\bx}}{\ell_0}, \quad
  \hat{h}        = \frac{h}{h_0},        \quad
  \hat{\bu} = \frac{\bu}{U_0}, \quad
  \hat{t}        = \frac{t}{t_0},        \quad
  \hat{b}        = \frac{b}{h_0},
  \end{equation}
with which, the SWEs \eqref{S1_E1} can be non-dimensionalized as follows:
\begin{equation}
\left\{
\begin{array}{ll}
\sr\, h_t + \nabla \cdot (h\bu)   = 0, \\ [3mm]
\sr\, (h\bu)_t + \nabla \cdot (h\bu\otimes \bu) + \frac{1}{\fr^2}\nabla(h^2/2) =-\frac{1}{\fr^2}h\nabla b,
\end{array}\right.
\label{S1_E3}
\end{equation}
where we drop the hat of the dimensionless variables for ease of presentation. The Strouhal number $\sr$ and
the Froude number $\fr$ are defined as
\begin{equation}
\sr := \frac{\ell_0}{t_0U_0}, \quad \fr:= \frac{U_0}{\sqrt{gh_0}}.
\end{equation}
In case of low Froude number flows, for which the flow velocities are systematically small as compared to the velocity of gravity waves,
a reference asymptotic expansion parameter $\eps$ can be introduced according
to the Froude number, via
\begin{equation}
\fr=\eps^\alpha \ll 1,
\end{equation}
with $\alpha$ chosen depending on the particular flow regime to be considered \cite{klein2005}.

In this work, we focus on flows over advective time scale where $\sr=1$, and assume $\fr=\eps$, namely $\alpha=1$ for the inviscid balanced flow over the topography, so that the dimensionless equations \eqref{S1_E3} become
  \begin{equation}
  \label{S1_E4}
  \left\{
  \begin{array}{ll}
  h_t + \nabla \cdot (h\bu)   = 0, \\ [3mm]
  (h\bu)_t + \nabla \cdot (h\bu\otimes \bu) + \frac{1}{\eps^2}\nabla(h^2/2) =  - \frac{1}{\eps^2}h\nabla b.
  \end{array}\right.
  \end{equation}
The system is hyperbolic, and its eigenvalues in the direction ${\bf{n}}$ are $\lambda_1 = \bu\cdot{\bf{n}} + c/\eps$ and $\lambda_2 = \bu\cdot{\bf{n}} - c/\eps$, with $c=\sqrt{h}$ being the scaled speed of sound.

One could directly apply the well-balanced shock capturing schemes to the dimensionless system \eqref{S1_E4}, however, due to the fact that the characteristic speed $\lambda_{1,2}$ is inversely proportional to the Froude number $\eps$, the time step constraint of an explicit time discretization satisfies
\[
\Delta t = \text{CFL} \frac{\Delta x}{\max(|\bu|+c/\eps)}\sim \eps \Delta x,
\]
where $\Delta t$ is the time step size,  $\Delta x$ is the mesh size and CFL is the time stability CFL number. As the Froude number $\eps$ approaching to zero, this leads to the stiffness in time, which is the same as the low Mach flows, see e.g. \cite{haack2012all,degond2011all,cordier2012asymptotic}. For low Mach flows, preconditioning techniques are usually applied to release the small time step condition and cure large numerical viscosities in the shock capturing schemes  \cite{Turkel1987,viozat1997implicit,dellacherie2010analysis,Ropke2015Mach,chen2020}. Such techniques, however, are effectively applicable only if the Mach numbers are not too small. On the other hand, naive implicit time discretizations of these shock capturing schemes result in the fully nonlinear systems, which are very inefficient to solve and sometimes may not be able to converge to the correct asymptotic limit.

In between, many semi-implicit schemes are developed, e.g., for low Mach (all Mach) Euler and Navier-Stokes equations \cite{degond2011all, haack2012all, tang2012, dimarco2017study, boscarino2018all,%BoRuSc18,
boscarino2019high,cordier2012asymptotic,tavelli2017pressure,Denner2018,Jonas2020novel,Denner2020,TPK2020,busto2021semi,boscarino2021high}, and for low Froude shallow water equations \cite{giraldo2010a,giraldo2013,tumolo2013,bispen2014imex,vater2018semi,liu2020well}, and many references therein. Among them, one type of methods which can ensure the correct asymptotic limit is the asymptotic preserving (AP) scheme. AP schemes were originally introduced in \cite{jin1999efficient} for multiscale kinetic equations, namely, the discretized scheme for a stiff PDE can converge to a consistent discretization of its limiting equation, under unresolved time step and mesh sizes, with uniform stability. For a recent review of AP schemes and their applications, see \cite{hu2017ap}. For Euler or Navier-Stokes equations will all-Mach number, AP schemes can well capture its corresponding incompressible limit as the Mach number approaching to zero \cite{degond2011all,haack2012all,cordier2012asymptotic,noelle2014weakly,bispen2017asymptotic,boscarino2019high}. AP schemes have also been applied to the shallow water equations with the low Froude number limit \cite{bispen2014imex,duran2015,couderc2017,liu2019asymptotic,liu2020well}. However, for the shallow water equations with an irregular bottom topography, most of current schemes either designed focus only in the low Froude number regime, or with up to second order accuracy.

In this paper, we propose high order well-balanced asymptotic preserving weighted essentially non-oscillatory (WENO) schemes for the shallow water equations with a non-flat bottom topography and all Froude numbers. For an irregular bottom with $b(\bx)\ne0$, the nonzero flux gradient and the source term are both scaled by the Froude number in \eqref{S1_E4}. Here, considering the still water equilibrium \eqref{still}, it is important to preserve both the well-balanced property and the low Froude limit for the dimensionless system \eqref{S1_E4}. A close work along this line is the second order well-balanced asymptotic preserving scheme developed and carefully analyzed by Liu in \cite{liu2020well}, which is based on a split system for the pre-balanced shallow water equations, following the idea for all Mach flows in \cite{haack2012all}. In our work, we will combine the high order AP schemes developed for the isentropic Euler and full Euler systems with all Mach numbers \cite{boscarino2019high,boscarino2021high} in the spirit of \cite{degond2011all,tang2012}, with the well balanced finite difference WENO schemes \cite{xing2005high}, to achieve a high order well-balanced asymptotic preserving scheme for the shallow water equations with a source term. We start by constructing a first order semi-implicit scheme. Similar to the hydrostatic pressure $p_2$ introduced for the pressure in the all-Mach flow \cite{boscarino2019high,boscarino2021high}, here an $H_2$ term corresponding to the variation from a constant water surface level with respect to the total water surface $H=h+b$ is introduced.
We first solve $H_2$ from an elliptic (or Helmholtz) equation, which is formed from a semi-implicit time discretization. After evaluating $H_2$, we can update the momentum $h\bu$ and then $h$. In this way, by utilizing a well balanced flux reconstruction in the updating of $h$ and $h\bu$, we can show that our first order semi-discrete scheme achieves the well balanced and AP properties at the same time. With the aid of a multi-stage explicit-implicit (IMEX) Runge-Kutta time discretization for a partitioned autonomous system, high order semi-implicit schemes can be obtained. Corresponding suitable high order spatial discretizations can also be constructed. Specifically high order well-balanced finite difference WENO reconstruction \cite{xing2005high} for convection terms are used in this paper, with high order central difference discretizations of second order and mixed derivatives in the elliptic (or Helmholtz) equation of $H_2$. The resulting high order semi-implicit scheme is showed to satisfy the well-balanced, AP and asymptotically accurate (AA) properties simultaneously, namely, the scheme is not only consistent (AP property) but also preserves the order of accuracy in time (AA property) in the stiff limit as $\eps\rightarrow 0$ \cite{Lorenzo2005}.

The rest of the paper is organized as follows. In Section \ref{sec2}, the low Froude limit of the SWEs is revisited. In Section \ref{sec3}, a well balanced AP scheme based on a first order semi-implicit scheme is first described and then generalized to high order methods. The analysis of well-balanced property, as well as AP and AA properties, follows afterward. Numerical experiments are presented in Section \ref{sec4}, which demonstrate the good performance of the high order well-balanced AP scheme in nearly hydrostatic flows and for a range of the Froude numbers including the zero Froude number limit. Conclusions are made in Section \ref{sec5}.

\section{Low Froude number limit for SWEs}
\label{sec2}
\setcounter{equation}{0}
\setcounter{figure}{0}
\setcounter{table}{0}

%%%%%%%%%%%%%%%%%%%%%%%%%%%%%%%%%%%%%%%%%%%%%%%%%%%%%%%%%%%%%%%%%%%%%%%%%%%
%%%%%%%%%%%%%%%Part 1
Let us denote $H=h+b$ as the water surface level, and the system \eqref{S1_E4} can be written as
  \begin{equation}
  \label{S2_E1}
\left\{
\begin{array}{ll}
  h_t + \nabla \cdot (h\bu)   = 0,\\ [3mm]
  (h\bu)_t + \nabla \cdot (h \bu \otimes \bu) + \frac{1}{\eps^2}h\nabla H = 0.
\end{array}\right.
  \end{equation}
We start with the following single-scale expansions of the solutions $h$  and $\bu$, in terms of $\eps$,
\begin{equation}
  \label{S2_E2}
\left\{
\begin{array}{ll}
  h(\bx,t) = h_0(\bx,t) + \eps h_1(\bx,t) +  \eps^2 h_2(\bx,t) + \cdots,\\ [3mm]
  \bu(\bx,t) = \bu_0(\bx,t) + \eps \bu_1(\bx,t) + \eps^2 \bu_2(\bx,t) + \cdots.
\end{array}\right.
\end{equation}
Since $H=h+b$ with $b=b(\bx)$ being time independent, we have
\begin{equation}
\label{S2_H}
H(\bx,t) = h_0(\bx,t)+ b(\bx) + \eps h_1(\bx,t) +  \eps^2 h_2(\bx,t) + \cdots
\end{equation}
Substituting (\ref{S2_E2}) and \eqref{S2_H} into (\ref{S2_E1}), equating to zero
for different orders of $\eps$, we have
\bit
\item $\mathcal{O}(\eps^{-2})$
  \begin{equation}
  \label{S2_E41}
  h_0 \nabla (h_0+b) =0,
  \end{equation}
  \item $\mathcal{O}(\eps^{-1})$
  \begin{equation}
  \label{S2_E42}
  h_1 \nabla (h_0+b) + h_0\nabla h_1 =0,
  \end{equation}
  \item $\mathcal{O}(\eps^{0})$
  \begin{equation}
  \label{S2_E43}
  \left\{
  \begin{array}{ll}
   (h_0)_t + \nabla \cdot (h_0\bu_0)   = 0,\\ [3mm]
   (h_0\bu_0)_t + \nabla \cdot (h_0 \bu_0 \otimes \bu_0) + h_2\nabla(h_0+b)+h_1\nabla h_1+ h_0\nabla h_2 = 0.
  \end{array}\right.
  \end{equation}
  \eit
  {Here for simplicity, no dry area is considered to exist in the domain so that $h_0\ne0$. Therefore, from \eqref{S2_E41} we obtain}
  \begin{equation}
  \label{S2_E5}
  h_0+b \equiv H_0(t),
  \end{equation}
  namely, $h_0+b$ is constant in space. It follows from \eqref{S2_E42} that $h_1\equiv H_1(t)$ is also constant in space. Since the bottom topography $b$ is assumed to be time independent, from \eqref{S2_E43}, we have
    \begin{subequations}
  \label{S2_E6}
  \begin{equation}
    \label{S2_E6a}
  \nabla \cdot (h_0\bu_0)   = -\frac{dH_0(t)}{dt},
  \end{equation}
  \begin{equation}
    \label{S2_E6b}
  (h_0\bu_0)_t + \nabla \cdot (h_0 \bu_0 \otimes \bu_0) + h_0\nabla h_2 = 0.
  \end{equation}
  \end{subequations}
   Now integrating the equation (\ref{S2_E6a}) over the spatial domain $\Omega$, it yields
  \begin{equation}
\label{S2_E7a}
\frac{dH_0(t)}{dt}=-\frac{1}{|\Omega|}\int_\Omega\nabla \cdot (h_0\bu_0)d\sigma=-\frac{1}{|\Omega|}\int_{\partial\Omega} h_0\bu_0\cdot\bn ds,
\end{equation}
where ${\bf n}$ is the unit outward normal vector along $\partial \Omega$, namely the time change of the total water height is given by the total flux of water across the domain boundary. \eqref{S2_E6} and \eqref{S2_E7a} form the classical zero Froude number shallow water equations, also known as the ``lake equations" \cite{greenspan1968theory,klein2005}.  If considering the no-slip $\bu\cdot \bn = 0$ or periodic boundary conditions, we further get $\int_{\Omega} \nabla \cdot (h_0 \bu_0) \,d{\bf x} = \int_{\partial \Omega} h_0 \bu_0 \cdot {\bf n}\,ds= 0$. This implies $H_0$ is constant both in space and time, i.e. $H_0 = $ Const. The same conclusion can also be derived for $H_1$. Therefore, the ``lake equations" further reduce to:
\begin{equation}\label{lake}
  \left\{
  \begin{array}{l}
  \nabla \cdot (h_0 \bu_0) = 0, \quad h_0+b=H_0=\text{Const.}\\ [3mm]
  \displaystyle \partial_t (h_0\bu_0) + \nabla \cdot \left( h_0 \bu_0 \otimes \bu_0 \right) + h_0\nabla h_2 = 0.
  \end{array}
  \right.
\end{equation}
A rigrous convergence analysis for the zero Froude limit from \eqref{S2_E1} to \eqref{lake} is very demanding, and we refer to \cite{klainerman1981singular,klainerman1982compressible} for such a rigorous study in the low Mach limit.

\section{Numerical schemes}
\label{sec3}
\setcounter{equation}{0}
\setcounter{figure}{0}
\setcounter{table}{0}

In this section, we will construct and analyze a class of high order finite difference schemes with the AP and well-balanced properties for the shallow water equations \eqref{S2_E1} with a range of Froude numbers. The SWEs in the form of \eqref{S2_E1} are very close to the isentropic Euler equations with all-Mach number, see e.g. \cite{degond2011all,boscarino2019high}.
However, it differs in the zero Froude limit, where in the isentropic Euler system, $\rho_0$ (corresponding to $h_0$ here) is constant, and it has the divergence free velocity field $\nabla\cdot \bu_0=0$. Here, $h_0$ is not a constant directly due to the appearance of source term. For the isentropic Euler equation in the zero Mach limit, the hydrodynamic pressure $p_2$ (corresponding to $H_2$ here) plays a role as a Lagrangian multiplier to ensure the divergence free condition, fortunately $H_2$ performs similarly in this setting. In the following, we will extend the high order semi-implicit finite difference WENO schemes developed in \cite{boscarino2019high} to solve \eqref{S2_E1}. We will analyze that the scheme can capture the zero Froude number shallow water equations, or the lake equations \eqref{lake} with no-slip or periodic boundary conditions, namely the scheme is asymptotic preserving.

For the shallow water equations with non-flat bottom topography, the well-balanced property is another important one, especially for capturing small perturbations of a still water equilibrium \cite{xing2005high}. We will adopt the well-balanced finite difference WENO reconstruction technique as developed in \cite{xing2005high}, tailored to our semi-implicit time discretization. We will show that under our semi-implicit framework, the well-balanced property can also be obtained.

\subsection{First order semi-implicit scheme}
We start with presenting a first order semi-implicit time discretization, while keeping space continuous at this moment. %Following the work for the isentropic Euler equations \cite{boscarino2019high},
The first order semi-implicit implicit-explicit (IMEX) scheme for \eqref{S2_E1} is given as follows
\begin{equation}
\label{S3_E1}
\left\{
  \begin{aligned}
  &\frac{h^{n+1} - h^n}{\Delta t} + \nabla \cdot (h\bu)^{n+1}   = 0,\\
  &\frac{(h\bu)^{n+1} -(h\bu)^n}{\Delta t} + \nabla \cdot\left(\frac{h\bu\otimes h\bu}{h}\right)^n  + \frac{1}{\eps^2}h^{n+1}\nabla H^{n+1} =0.
  \end{aligned}
  \right.
\end{equation}
Notice that $H^{n+1}=h^{n+1}+b$. From the second equation of (\ref{S3_E1}), we can first express $(h\bu)^{n+1}$ in terms of $h^{n+1}$, $H^{n+1}$ and other variables at time level $t^n$. Substituting it into the first equation of (\ref{S3_E1}), we get
\begin{equation}
\label{S3_E2}
\left\{
  \begin{aligned}
  &\frac{h^{n+1} - h^n}{\Delta t} + \nabla \cdot (h\bu)^n -\Delta t \nabla ^2:\left(\frac{h\bu\otimes h\bu}{h}\right)^n
  -\frac{\Delta t}{\eps^2}\nabla\cdot(h^{n+1}\nabla H^{n+1}) =0,\\
  &\frac{(h\bu)^{n+1} -(h\bu)^n}{\Delta t} + \nabla \cdot\left(\frac{h\bu\otimes h\bu}{h}\right)^n  + \frac{1}{\eps^2}h^{n+1}\nabla H^{n+1} =0,
  \end{aligned}
  \right.
\end{equation}
where $:$ is the tensor double dot product. The first equation of \eqref{S3_E2} appears to be a nonlinear system for $h^{n+1}$, as $H^{n+1}=h^{n+1}+b$. To avoid it, a slight modification of the first equation yields
\begin{equation} \label{eq3.3}
\frac{h^{n+1} - h^n}{\Delta t} + \nabla \cdot (h\bu)^n -\Delta t \nabla ^2:\left(\frac{h\bu\otimes h\bu}{h}\right)^n
  -\frac{\Delta t}{\eps^2}\nabla\cdot(h^n\nabla H^{n+1}) = 0,
\end{equation}
which is now a linear equation for the unknown function $h^{n+1}$. Similarly, the corresponding term in the second equation of \eqref{S3_E2} can be replaced by $\frac{1}{\eps^2}h^n\nabla H^{n+1}$, which is easier for the generalization to high order described in the following subsection.

To deal with the stiff diffusive term $\frac{1}{\eps^2}\nabla\cdot(h^n\nabla H^{n+1})$, in this work, we consider the no-slip or periodic boundary conditions, namely, $h_0+b=H_0$ and $h_1=H_1$ are both constants in the asymptotic expansion \eqref{S2_H}. We may now introduce a water surface perturbation $H_2$, defined as
\begin{equation}
\label{H2}
H_2 = \frac{H-\bar{H}}{\eps^2},
\end{equation}
where $\bar{H}$ denotes the spatial average of the water surface level $H$ (computed from $h+b$). In this way, the term $H_2$ in the zero Froude shallow
water limit converges to $h_2$, which remains finite. Numerically, we take $\bar{H}$ as the spatial average of $H^n$, that is
\begin{equation}
\label{S3_Hbar}
H^{n+1}=\bar{H}^n + \eps^2 H^{n+1}_2 \quad \text{ and }\quad h^{n+1}=\bar{H}^n-b+\eps^2 H^{n+1}_2,
\end{equation}
so that we obtain a linear elliptic equation for $H^{n+1}_2$ from \eqref{eq3.3}
\begin{equation}
\label{S3_H2}
\eps^2 H^{n+1}_2 -\Delta t^2\nabla\cdot(h^n\nabla H^{n+1}_2) = h^*,
\end{equation}
with
\begin{equation}
\label{S3_hstar}
h^*=H^n-\bar{H}^n-\Delta t \left(\nabla \cdot (h\bu)^n -\Delta t \nabla ^2:\left(\frac{h\bu\otimes h\bu}{h}\right)^n\right).
\end{equation}
After obtaining $H^{n+1}_2$ from \eqref{S3_H2}, $h^{n+1}$ and $H^{n+1}$ can be updated from \eqref{S3_Hbar}. In 
equation \eqref{S3_E2}, we can replace $1/\eps^2h^{n+1}\nabla H^{n+1}$ by $h^n\nabla H^{n+1}_2$, which leads to
\begin{equation}
\label{S3_hu}
\left\{
\begin{aligned}
&\frac{h^{n+1} - h^n}{\Delta t} + \nabla \cdot (h\bu)^n -\Delta t \nabla ^2:\left(\frac{h\bu\otimes h\bu}{h}\right)^n
-\nabla\cdot(h^{n}\nabla H_2^{n+1}) =0,\\
&\frac{(h\bu)^{n+1} -(h\bu)^n}{\Delta t} + \nabla \cdot\left(\frac{h\bu\otimes h\bu}{h}\right)^n  + h^{n}\nabla H_2^{n+1} =0.
\end{aligned}
\right.
\end{equation}
We can solve the second equation of \eqref{S3_hu} for $(h\bu)^{n+1}$.
This semi-implicit treatment in time can ensure the right asymptotic limit as the Froude number $\eps\rightarrow 0$, which
is known as the AP property. We will analyze it afterward.
In general, direct computing $h^{n+1}$ from \eqref{S3_Hbar} cannot preserve exact mass conservation, and we may further update $h^{n+1}$ using the first equation of \eqref{S3_E1} with the available $(h\bu)^{n+1}$.

Next we will discuss the spatial discretizations according to the first order semi-implicit time discretization. The main guidance is to preserve the equilibrium state for a still water when $H=h+b=$ Const. and $h\bu ={\bf 0}$, and also avoid excessive numerical viscosity inversely proportional to the Froude number $\eps$. %Such schemes are named well-balanced (WB) schemes or schemes satisfying C-property \cite{xing2005high,noelle2009high}.
We follow both the well-balanced finite difference scheme developed in \cite{xing2005high}, and the spatial discretizations for the all-Mach isentropic Euler equations in \cite{boscarino2019high}. First or second low order discretizations will be described first, and high order extensions will be presented afterward.

To preserve the still water equilibrium for the water surface level $H$, it is more convenient to rewrite the first equation of \eqref{S3_E1} in a pre-balanced form, namely
\begin{equation}
\label{S3_H}
\frac{H^{n+1} - H^n}{\Delta t} + \nabla \cdot (h\bu)^{n+1} = 0,
\end{equation}
which is equivalent to the original equation since $b$ is independent of time.
In the case of still-water equilibrium \eqref{still}, to preserve the water surface level $H=$ Const, it requires that no numerical viscosity should be presented in the numerical approximation of the flux term $\nabla \cdot (h\bu)^{n+1}$. Therefore, the numerical viscosity term should depend on $H$ instead of $h$, and a Lax-Friedrichs flux for $\nabla \cdot (h\bu)$ (we drop the superindex $n+1$ for brevity) is defined as follows:
\begin{subequations}
\label{S3_LF}
\begin{equation}
  (\widehat{hu})_{i+\frac{1}{2},j} =
\frac{1}{2}\left[(hu)_{i+1,j} + (hu)_{i,j} -
\alpha^x_{i,j}(H_{i+1,j} - H_{i,j})\right],
\end{equation}
\begin{equation}
(\widehat{hv})_{i,j+\frac{1}{2}} =
\frac{1}{2}\left[(hv)_{i,j+1} + (hv)_{i,j} -
\alpha^y_{i,j}(H_{i,j+1} - H_{i,j})\right],
\end{equation}
\end{subequations}
where $\alpha^x_{i,j}=\max_{h,hu}(|u|+\min(1,1/\eps)\sqrt{h})$ and $\alpha^y_{i,j}=\max_{h,hv}(|v|+\min(1,1/\eps)\sqrt{h})$ are the local viscosity coefficients along $x$ and $y$ respectively.
We denote
\begin{equation}
\label{S3_DLF}
\nabla_{LF}\cdot (h\bu) = \frac{(\widehat{hu})_{i+\frac{1}{2},j}-(\widehat{hu})_{i-\frac{1}{2},j}}{\Delta x}+\frac{(\widehat{hv})_{i,j+\frac{1}{2}}-(\widehat{hv})_{i,j-\frac{1}{2}}}{\Delta y}.
\end{equation}
It is easy to see $\nabla_{LF}\cdot (h\bu) = 0$ if $H=h+b=$ Const. and $\bu=(u,v)=(0,0)$.

For the second equation of \eqref{S3_hu}, it does not matter very much how to choose the numerical fluxes since the numerical viscosity depends on $h\bu$ which is $0$ for still water, e.g., a local Lax-Friedrichs flux for the second term $\nabla \cdot \left(h\bu\otimes h\bu/h\right)$ is
\begin{subequations}
\label{S3_LFhu}
\begin{equation}
\left (\widehat{hu^2}\right)_{i+\frac{1}{2},j} =
\frac{1}{2}\left[\left(hu^2\right)_{i+1,j} + \left(hu^2\right)_{i,j} -
\alpha^x_{i,j}\Big((hu)_{i+1,j} - (hu)_{i,j}\Big)\right],
\end{equation}
\begin{equation}
\left (\widehat{huv}\right)_{i,j+\frac{1}{2}} =
\frac{1}{2}\left[\left(huv\right)_{i,j+1} + \left(huv\right)_{i,j} -
\alpha^y_{i,j}\Big((hu)_{i,j+1} - (hu)_{i,j}\Big)\right],
\end{equation}
\begin{equation}
\left (\widehat{huv}\right)_{i+\frac{1}{2},j} =
\frac{1}{2}\left[\left(huv\right)_{i+1,j} + \left(huv\right)_{i,j} -
\alpha^x_{i,j}\Big((hv)_{i+1,j} - (hv)_{i,j}\Big)\right],
\end{equation}
\begin{equation}
\left (\widehat{hv^2}\right)_{i,j+\frac{1}{2}} =
\frac{1}{2}\left[\left(hv^2\right)_{i,j+1} + \left(hv^2\right)_{i,j} -
\alpha^y_{i,j}\Big((hv)_{i,j+1} - (hv)_{i,j}\Big)\right],
\end{equation}
\end{subequations}
and $\alpha^x_{i,j}$ and $\alpha^y_{i,j}$ are the local viscosity coefficients which can be taken the same as above.

Notice that preserving the still water equilibrium in \eqref{S3_hu} is to require $(h\bu)^{n+1}={\bf 0}$, which can be satisfied from requiring $h^n\nabla H^{n+1}_2 = 0$. However, a straightforward numerical discretization of the term $h\nabla H_2$ may lead to a nonconservative discretization, even in the special case of $b=0$ when such term should be treated in the conservative manner. This issue has been addressed in the well-balanced WENO methods studied in \cite{xing2005high}. By adopting such idea to decompose the source term, at the continuous level using the relation \eqref{H2} we can rewrite this term as
\beq
\label{presplit}
h\nabla H_2 = \frac{1}{\eps^2} h\nabla H = \frac{1}{\eps^2} \nabla \left(\frac12 h^2 - \frac12 b^2 \right) + \frac{1}{\eps^2} H\nabla b = \nabla\left(\bar H H_2+\frac12\eps^2 H^2_2-H_2b\right)+H_2\nabla b.
\eeq
Taking $H_2$ as $H^{n+1}_2$ and using central differences for a low order spatial discretization for both terms, that is 
$$
\nabla_C\bigg(\bar H^n H^{n+1}_2+\frac12\eps^2 (H^{n+1}_2)^2-H^{n+1}_2 b\bigg)+H^{n+1}_2\nabla_C b,
$$ 
which is still $0$ when $H_2^{n+1}\equiv 0$.

%\YX{When $b=0$, i.e., no source term, will the second equation have mass conservation? In other words, is the proposed method conservative, hence satisfying the Lax-Wendroff theorem?}
%\YX{If not, we can try to use the source term decomposition idea as in \cite{xing2005high}.}

The remaining spatial discretizations for \eqref{S3_H2} and \eqref{S3_hstar} are as follows. We use central difference discretization for the second order derivatives terms, denoted with subindex $C$, where
\begin{equation}
\begin{aligned}
\nabla^2_C:\left(\frac{h\bu\otimes h\bu}{h}\right)
=& \frac{1}{\Delta x^2}\Big[(hu^2)_{i+1,j}-2(hu^2)_{i,j} + (hu^2)_{i-1,j}\Big]\\
&+\frac{1}{2\Delta x \Delta y}\Big[\Big((huv)_{i+1,j+1} - (huv)_{i-1.j+1}\Big)
- \Big((huv)_{i+1,j-1} - (huv)_{i-1,j-1}\Big)\Big]\\
&+\frac{1}{\Delta y^2}\Big[(hv^2)_{i,j+1}-2(hv^2)_{i,j} + (hv^2)_{i,j-1}\Big].
\end{aligned}
\end{equation}
For the term $\nabla \cdot(h^n\nabla H^{n+1})$ expressing in the form
\[
\nabla \cdot(h^n\nabla H^{n+1})
=\partial_x(h^n\partial_xH^{n+1})
+\partial_y(h^n\partial_yH^{n+1}),
\]
we may take a compact central difference for terms like $(a(x,y)q_x)_x$ at the grid point $(x_i,y_j)$
\[
(a(x,y)q_x)_x\Big|_{(x_i, y_j)}
=\frac{1}{\Delta x^2}
(a_{i-1,j}, a_{i,j}, a_{i+1,j})
\begin{pmatrix}
\frac{1}{2} & -\frac{1}{2} & 0\\
\frac{1}{2} & -1           &\frac{1}{2}\\
0           & -\frac{1}{2} & \frac{1}{2}
\end{pmatrix}
\begin{pmatrix}
q_{i-1,j}\\
q_{i,j}    \\
q_{i+1,j}
\end{pmatrix},
\]
and similar approximation can be done for $(a(x,y)q_y)_y$ along the $y$ direction. This will form a positive definite linear system for the left side of \eqref{S3_H2}, if $h^n$ keeps positive. We denote the numerical approximation of $\nabla \cdot(h^n\nabla H^{n+1})$ by $\nabla_{C^2} \cdot(h^n\nabla H^{n+1})$. Lastly
$\nabla\cdot(h\bu)^n$ is discretized the same as in \eqref{S3_DLF}, so that on the right side of \eqref{S3_H2}
\begin{equation}
\label{S3_hstar2}
h^*=H^n-\bar{H}^n-\Delta t \left(\nabla_{LF} \cdot (h\bu)^n -\Delta t \nabla_C ^2:\left(\frac{h\bu\otimes h\bu}{h}\right)^n\right),
\end{equation}
which is clearly $0$ for $H=$ Const. and $h\bu = {\bf 0}$. With such discretizations, solving $H^{n+1}_2$ from \eqref{S3_H2} yields $H^{n+1}_2\equiv 0$, so the well-balanced property for the still water is well preserved.

We now summarize the first order semi-implicit scheme as follows:
\begin{equation}
\label{S3_1st}
\left\{
\begin{aligned}
&\eps^2 H^{n+1}_2 -\Delta t^2\nabla_{C^2}\cdot(h^n\nabla H^{n+1}_2) =H^n-\bar{H}^n-\Delta t \left(\nabla_{LF} \cdot (h\bu)^n -\Delta t \nabla ^2_C:\left(\frac{h\bu\otimes h\bu}{h}\right)^n\right),\\
%& H^{n+1}=\bar{H}^n + \eps^2 H^{n+1}_2 \quad \text{ and }\quad h^{n+1}=\bar{H}^n-b+\eps^2 H^{n+1}_2, \\
&\frac{(h\bu)^{n+1} -(h\bu)^n}{\Delta t} + \nabla_{LF} \cdot\left(\frac{h\bu\otimes h\bu}{h}\right)^n  + \nabla_C\bigg(\bar H^n H^{n+1}_2+\frac12\eps^2 (H^{n+1}_2)^2-H^{n+1}_2 b\bigg)+H^{n+1}_2\nabla_C b =0, \\
&\frac{h^{n+1} - h^n}{\Delta t} + \nabla_{LF} \cdot (h\bu)^{n+1}   = 0,
\end{aligned}
\right.
\end{equation}
which is performed in a sequential way.

%%%%%%%%%%%%%%%%%%%%%%%%%%%%%%%%%%%%%%%%%
\subsection{High order semi-implicit scheme}
To extend the first order semi-implicit scheme to high order, we follow a similar procedure as described in \cite{boscarino2019high,boscarino2018all}. For ease of presentation, we keep space continuous first. Let's write \eqref{S2_E1} as an autonomous system
\begin{equation}
\label{S3_E5}
U_t = \mathcal{H} (U,U),
\end{equation}
where $U=(h,h\bu)^T$ and $\mathcal{H}: \mathbb{R}^n \times \mathbb{R}^n \to \mathbb{R}^n $ is a sufficiently regular mapping. We use two different arguments
for $U$ with different treatments, one is explicit with subindex ``E" and the other is implicit with subindex ``I", that is $U_E=(h_E, (h\bu)_E)^T$ and $U_I=(h_I, (h\bu)_I)^T$, and we solve
\begin{equation}\label{PartitionedSyst}
\left\{
\begin{array}{l}
U'_E =  \mathcal{H}(U_E,U_I),\\[3mm]
U'_I =  \mathcal{H}(U_E,U_I),
\end{array}
\right.
\end{equation}
where we define
\begin{equation}
	\label{S3_E6}
	\mathcal{H}(U_E,U_I)
	=
	\begin{pmatrix}
	-\nabla \cdot (h\bu)_I,  \\ \, \\
	-\nabla \cdot \left(h\bu\otimes \bu\right)_E
	-h_E\nabla H_{I,2}
	\end{pmatrix}.
\end{equation}
$H_{I,2}$ is defined similarly as in \eqref{H2}
\begin{equation}
H_{I,2}=\frac{H_I-\bar{H}_E}{\eps^2}=\frac{h_I+b-\bar{H}_E}{\eps^2},
\end{equation}	
and $\bar{H}_E$ is the spatial average of $h_E+b$. For the first order semi-implicit scheme, $U_E=U^n=(h^n,(h\bu)^n)^T$ and $U_I=U^{n+1}=(h^{n+1},(h\bu)^{n+1})^T$.

For the partitioned system (\ref{PartitionedSyst}), we need to apply an IMEX Runge-Kutta time discretization with a double Butcher $tableau$ \cite{butcher2016},
\begin{equation}\label{DBT}
\begin{array}{c|c}
\tilde{c} & \tilde{A}\\
\hline
\vspace{-0.25cm}
\\
& \tilde{b^T} \end{array} \ \ \ \ \ \qquad
\begin{array}{c|c}
{c} & {A}\\
\hline
\vspace{-0.25cm}
\\
& {b^T} \end{array},
\end{equation}
where $\tilde{A} = (\tilde{a}_{ij})$ is an $s \times s$ matrix for an explicit scheme, with $\tilde{a}_{ij}=0$ for $j \geq i$ and $A = ({a}_{ij})$ is an  $s \times s$ matrix for an implicit scheme. For the implicit part of the methods, we use a diagonally implicit scheme, i.e. $a_{ij}=0$, for $j > i$, in order to guarantee simplicity and efficiency in solving the algebraic equations corresponding to the implicit part of the discretization. The vectors $\tilde{c}=(\tilde{c}_1,...,\tilde{c}_s)^T$, $\tilde{b}=(\tilde{b}_1,...,\tilde{b}_s)^T$, and $c=(c_1,...,c_s)^T$, $b=(b_1,...,b_s)^T$ complete the characterization of the scheme. The coefficients $\tilde{c}$ and $c$ are given by the usual relation
\begin{eqnarray}\label{eq:candc}
\tilde{c}_i = \sum_{j=1}^{i-1} \tilde a_{ij}, \ \ \ c_i = \sum_{j=1}^{i} a_{ij}.
\end{eqnarray}
For the first order semi-implicit scheme, it corresponds to $s=1$, and the double Butcher Tableau is
\begin{equation*}\label{Afirst}
\begin{array}{c|c}
0 & 0 \\
\hline
&1
\end{array} \qquad\qquad
\begin{array}{c|c}
1 & 1     \\
\hline
&  1
\end{array},
\end{equation*}
namely $U_E=U^n$ and $U_I=U^{n+1}$.

For a high order semi-implicit scheme, a multi-stage IMEX Runge-Kutta is needed, usually it is characterized as the triplet $(s, \sigma, p)$, for the number of stages of the implicit scheme ($s$), the number of stages of the explicit scheme ($\sigma$) and the order of the scheme ($p$). Here we adopt the IMEX scheme as constructed in \cite{boscarino2021high}, which we require
$\sigma=s$ with $s$ stages for both implicit and explicit parts, and $\tilde{c}_i=c_i$ for $i=2,\cdots,s$.

Now, we may update the solutions as follows. Starting from $U_E^{(0)}=U_I^{(0)}=U^n$, for inner stages $i = 1  \text{ to }  s$:
\bit
\item First update the solution $U_E^{(i)}$ for the explicit part
\begin{equation}
  \label{S3_E7}
  U_E^{(i)} = U^n + \Delta t\sum^{i-1}_{j=1}\tilde{a}_{ij}\mathcal{H}(U_E^{(j)},U_I^{(j)}).
\end{equation}
\item Update the known values for the implicit part $U_*^{(i)}$, where
  \begin{equation}
  \label{S3_E8}
  U_*^{(i)} = U^n + \Delta t\sum^{i-1}_{j=1}a_{ij}\mathcal{H}(U_E^{(j)},U_I^{(j)}),
  \end{equation}
and then solve
   \begin{equation}
   \label{S3_E9}
   U_I^{(i)} = U_*^{(i)} + \Delta ta_{ii}\mathcal{H}(U_E^{(i)},U_I^{(i)}).
   \end{equation}
\item Finally, the solution $U^{n+1}$ at time level $t^{n+1}$ is accumulated by
\begin{equation}
\label{S3_E15}
U^{n+1}  = U^n + \Delta t \sum_{i=1}^s b_i \mathcal{H}(U_E^{(i)},U_I^{(i)}).
\end{equation}
\eit
In components, the procedures corresponding to $U^{(i)}_E$ and $U^{(i)}_*$ are
\begin{subequations}
\label{S3_UE}
\begin{equation}
h_E^{(i)} = h^n - \Delta t\sum_{j=1}^{i-1}\tilde{a}_{ij}\nabla\cdot (h\bu)_I^{(j)},
\end{equation}
\begin{equation}
(h\bu)_E^{(i)} =(h\bu)^n - \Delta t\sum_{j=1}^{i-1}\tilde{a}_{ij}\left(\nabla \cdot \left(\frac{h\bu\otimes h\bu}{h}\right)_E^{(j)}+h^{(j)}_E\nabla H_{I,2}^{(j)}\right),
\end{equation}
\end{subequations}
\begin{subequations}
	\label{S3_Ustar}
	\begin{equation}
	h_*^{(i)} = h^n - \Delta t\sum_{j=1}^{i-1}a_{ij}\nabla\cdot (h\bu)_I^{(j)},
	\end{equation}
	\begin{equation}
	(h\bu)_*^{(i)} =(h\bu)^n - \Delta t\sum_{j=1}^{i-1}a_{ij}\left(\nabla \cdot \left(\frac{h\bu\otimes h\bu}{h}\right)_E^{(j)}+h^{(j)}_E\nabla H_{I,2}^{(j)}\right),
	\end{equation}
\end{subequations}
and for $U^{(i)}_I$ it takes the form
\begin{subequations}
    \label{S3_E11}
	\begin{equation}
    h_I^{(i)} = h_*^{(i)} - {a}_{ii}\Delta t\nabla\cdot (h\bu)_I^{(i)} ,
    \end{equation}
    \begin{equation}
    (h\bu)_I^{(i)} =(h\bu)_*^{(i)} -{a}_{ii}\Delta t\left(\nabla \cdot \left(\frac{h\bu\otimes h\bu}{h}\right)_E^{(i)}+h^{(i)}_E\nabla H_{I,2}^{(i)}\right).
	\end{equation}
\end{subequations}
In order to solve the implicit components in \eqref{S3_E11}, a similar fashion as in the first order case can be followed. By substituting $(h\bu)_I^{(i)}$ from the second equation into the first equation, replacing $h^{(i)}_I$ by $h^{(i)}_I=\bar{H}^{(i)}_E+\eps^2 H^{(i)}_{I,2}$, where $\bar{H}^{(i)}_E$ is the spatial average of $h^{(i)}_E+b$, we obtain
\begin{equation}
\label{S3_E12}
\eps^2 H^{(i)}_{I,2} -(a_{ii}\Delta t)^2\nabla\cdot\Big(h^{(i)}_E\nabla H^{(i)}_{I,2}\Big) = h^{**},
\end{equation}
with
\begin{equation}
\label{S3_E13}
h^{**}=h^{(i)}_*+b-\bar{H}^{(i)}_E-a_{ii}\Delta t\left(\nabla \cdot (h\bu)^{(i)}_* -a_{ii}\Delta t \nabla ^2:\left(\frac{h\bu\otimes h\bu}{h}\right)^{(i)}_E\right).
\end{equation}
Lastly, the equations \eqref{S3_E15} can be rewritten as
\begin{subequations}
	\label{S3_Un1}
\begin{equation}
h^{n+1} = h^n - \Delta t\sum_{i=1}^{s}b_i\nabla\cdot (h\bu)_I^{(i)},
\end{equation}
\begin{equation}
(h\bu)^{n+1} =(h\bu)^n - \Delta t\sum_{i=1}^{s}b_{i}\left(\nabla \cdot \left(\frac{h\bu\otimes h\bu}{h}\right)_E^{(i)}+h^{(i)}_E\nabla H_{I,2}^{(i)}\right).
\end{equation}
\end{subequations}

For high order in space, we will adopt the finite difference WENO reconstruction \cite{jiang1996efficient,shu1998essentially,shu2009high} for the first order convection terms, and central difference for the second order and mixed derivatives. The numerical fluxes for the convection terms are chosen in the same spirit as the first order case described above.

We take $(hu)_x$ in the convection term $\nabla\cdot(h\bu)$ as an example, and omit the indexes for brevity. For high order finite difference reconstruction, the flux needs to split into an upwind and another downwind part, e.g., for the Lax-Friedrichs flux splitting, we have
\begin{equation}
\label{splitflux}
(hu)^\pm_{i+\ell,j}=\frac12\Big((hu)_{i+\ell,j}\pm\alpha^x_{i,j}H_{i+\ell,j}\Big), \quad \ell=-r,\cdots,r,
\end{equation}
where $\alpha^x_{i,j}=\max_{h,hu}(|u|+\min(1,1/\eps)\sqrt{h})$ is the local numerical viscosity coefficient over the stencil $S=\{(i-r,j),\cdots,(i+r,j)\}$.
It is also important to take $H$ instead of $h$ in \eqref{splitflux}, in order to preserve the still water equilibrium.
The split fluxes can be used to reconstruct $(\widehat{hu})^\mp_{i\pm\frac12,j}$ based on upwind and downwind WENO reconstructions. In our numerical section, a fifth order finite difference WENO reconstruction with $r=2$ is use. The numerical flux for $(hu)_x$ is defined as
\begin{equation}
(\widehat{hu})_{i+\frac12,j}=(\widehat{hu})^-_{i+\frac12,j}+(\widehat{hu})^+_{i+\frac12,j}.
\end{equation}
The numerical flux $(\widehat{hv})_{i,j+\frac12}$ along the $y$ direction can be defined similarly.
With these numerical fluxes, the convection term $\nabla\cdot(h\bu)$ can be approximated by $\nabla_{LF}\cdot(h\bu)$ as defined in \eqref{S3_DLF}.
The term $\nabla \cdot \left(h\bu\otimes h\bu/h\right)$ in the momentum equation can be approximated in a similar way by the high order finite difference WENO reconstruction as $\nabla\cdot(h\bu)$, e.g., for $(hu^2)_x$ and $(huv)_y$ in the momentum equation of $hu$, a Lax-Friedrichs flux splitting is taken as
\begin{subequations}
\label{splitflux2}
\begin{equation}
(hu^2)^\pm_{i+\ell,j}=\frac12\Big((hu^2)_{i+\ell,j}\pm\alpha^x_{i,j}(hu)_{i+\ell,j}\Big), \quad \ell=-r,\cdots,r,
\end{equation}
\begin{equation}
(huv)^\pm_{i,j+\ell}=\frac12\Big((huv)_{i,j+\ell}\pm\alpha^y_{i,j}(hu)_{i,j+\ell}\Big), \quad \ell=-r,\cdots,r.
\end{equation}
\end{subequations}
Similarly for $(huv)_x$ and $(hv^2)_y$ in the momentum equation of $hv$, so that we get the approximation for  $\nabla \cdot \left(h\bu\otimes h\bu/h\right)$, which is still denoted as  $\nabla_{LF} \cdot \left(h\bu\otimes h\bu/h\right)$.

For the second order derivative terms appeared in $\nabla^2:(\frac{h\bu\otimes h\bu}{h})$, a high order central difference discretization is used, which is denoted as $\nabla^2_C:(\frac{h\bu\otimes h\bu}{h})$. In our numerical section, we take a fourth order central difference discretization.
For example, along the $x$ direction, we approximate $q_{xx}$ by
\[
q_{xx}|_{x=x_i} = \frac{-q_{i-2} + 16q_{i-1} -30q_{i} + 16q_{i+1} - q_{i+2}}{12\Delta x^2}+\mathcal{O}(\Delta x^4).
\]
For the mixed derivative term $q_{xy}$, it is discretized  dimension-by-dimension with a fourth order central difference scheme along each direction, e.g., along the $x$ direction
\[
q_x|_{x=x_i} = \frac{q_{i-2} - 8q_{i-1} + 8q_{i+1} - q_{i+2}}{12\Delta x^2} + \mathcal{O}(\Delta x^4).
\]

For the variable coefficient diffusion term $\nabla \cdot(h\nabla H)$, we take a compact fourth order central difference discretization as developed in \cite{boscarino2019high}, which is denoted as $\nabla_{C^2} \cdot(h\nabla H)$. Taking $(a(x,y)q_x)_x$ at the grid point $(x_i,y_j)$ as an example, it is approximated by
\begin{equation*}
  (a(x)q_x)_x|_{(x_i, y_j)}
  =\frac{1}{\Delta x^2}
  {\bf a}_{i,j}
  \begin{pmatrix}
  -25/144 & 1/3 & -1/4 & 1/9 & -1/48\\
    1/6   & 5/9 &  -1  & 1/3 & -1/18\\
     0    &  0  &   0  &  0  &   0  \\
   -1/18  & 1/3 &  -1  & 5/9 &  1/6 \\
   -1/48  & 1/9 & -1/4 & 1/3 &-25/144
  \end{pmatrix}
  {\bf q}^T_{i,j}
  +\mathcal{O} (\Delta x^4),
\end{equation*}
with the two vectors being
\begin{equation*}
  {\bf a}_{i,j}=(a_{i-2,j},a_{i-1,j},a_{i,j},a_{i+1,j},a_{i+2,j}), \quad
  {\bf q}_{i,j}=(q_{i-2,j},q_{i-1,j},q_{i,j},q_{i+1,j},q_{i+2,j}).
\end{equation*}

For high order spatial discretization, as in \eqref{presplit}, we rewrite the term $h\nabla H_2$ and discretize it as
\beq
\nabla_W\big(\bar H_E H_{I,2}+\frac12\eps^2 (H_{I,2})^2-H_{I,2}b\big)+H_{I,2}\nabla_W b,
\eeq
here $\nabla_W$ in the first term is a high order finite difference WENO reconstruction, but with zero viscosity as studied in \cite{boscarino2019high}. We use the same $\nabla_W$ with exactly the same nonlinear weights to evaluate $\nabla_W b$, for the purpose of preserving the exact still water equilibrium. We refer to \cite{xing2005high} for more detailed discussion of this matter.

With the above space and time discretizations, we now summarize our high order semi-implicit scheme as follows:
\bit
\item First for the stage values from $i=1,\cdots, s$:
\begin{enumerate}
\item update $h_E^{(i)}$ and $(h\bu)_E^{(i)}$ from
\begin{subequations}
	\label{S3_UE2}
	\begin{equation}
	h_E^{(i)} = h^n - \Delta t\sum_{j=1}^{i-1}\tilde{a}_{ij}\nabla_{LF}\cdot (h\bu)_I^{(j)},
	\end{equation}
	\begin{equation}
	(h\bu)_E^{(i)} =(h\bu)^n - \Delta t\sum_{j=1}^{i-1}\tilde{a}_{ij}\left(\nabla_{LF} \cdot \left(\frac{h\bu\otimes h\bu}{h}\right)_E^{(j)}+{\nabla_W\Big(\bar H^{(j)}_E H^{(j)}_{I,2}+\frac12\eps^2 (H^{(j)}_{I,2})^2-H^{(j)}_{I,2}b\Big)+H^{(j)}_{I,2}\nabla_W b}\right).
	\end{equation}
\end{subequations}
\item precompute the known values of $h_*^{(i)}$ and $(h\bu)_*^{(i)}$
\begin{subequations}
	\label{S3_Ustar2}
	\begin{equation}
	h_*^{(i)} = h^n - \Delta t\sum_{j=1}^{i-1}a_{ij}\nabla_{LF}\cdot (h\bu)_I^{(j)},
	\end{equation}
	\begin{equation}
	(h\bu)_*^{(i)} =(h\bu)^n - \Delta t\sum_{j=1}^{i-1}a_{ij}\left(\nabla_{LF} \cdot \left(\frac{h\bu\otimes h\bu}{h}\right)_E^{(j)}+{\nabla_W\Big(\bar H^{(j)}_E H^{(j)}_{I,2}+\frac12\eps^2 (H^{(j)}_{I,2})^2-H^{(j)}_{I,2}b\Big)+H^{(j)}_{I,2}\nabla_W b}\right).
	\end{equation}
\end{subequations}
\item solve the linear elliptic equation to obtain $H^{(i)}_{I,2}$
\begin{subequations}
\label{S3_HI2}
\begin{equation}
\label{S3_HI21}
\eps^2 H^{(i)}_{I,2} -(a_{ii}\Delta t)^2\nabla_{C^2}\cdot\Big(h^{(i)}_E\nabla H^{(i)}_{I,2}\Big) = h^{**},
\end{equation}
\begin{equation}
\label{S3_HI22}
h^{**}=h^{(i)}_*+b-\bar{H}^{(i)}_E-a_{ii}\Delta t\left(\nabla_{LF} \cdot (h\bu)^{(i)}_* -a_{ii}\Delta t \nabla ^2_C:\left(\frac{h\bu\otimes h\bu}{h}\right)^{(i)}_E\right).
\end{equation}
\end{subequations}
\item update $h_I^{(i)}$ and $(h\bu)_I^{(i)}$ from
\begin{subequations}
	\label{S3_UI2}
	\begin{equation}
	\label{S3_UI21}
	h_I^{(i)} = h_*^{(i)} - {a}_{ii}\Delta t\nabla_{LF}\cdot (h\bu)_I^{(i)} ,
	\end{equation}
	\begin{equation}
	\label{S3_UI22}
	(h\bu)_I^{(i)} =(h\bu)_*^{(i)} -{a}_{ii}\Delta t\left(\nabla_{LF} \cdot \left(\frac{h\bu\otimes h\bu}{h}\right)_E^{(i)}+{\nabla_W\Big(\bar H^{(i)}_E H^{(i)}_{I,2}+\frac12\eps^2 (H^{(i)}_{I,2})^2-H^{(i)}_{I,2}b\Big)+H^{(i)}_{I,2}\nabla_W b}\right).
	\end{equation}
\end{subequations}
\end{enumerate}
\item Update the solution at the time level $t^{n+1}$:
\begin{subequations}
	\label{S3_Unp1}
	\begin{equation}
	h^{n+1} = h^n - \Delta t\sum_{i=1}^{s}b_i\nabla_{LF}\cdot (h\bu)_I^{(i)},
	\end{equation}
	\begin{equation}
	(h\bu)^{n+1} =(h\bu)^n - \Delta t\sum_{i=1}^{s}b_{i}\left(\nabla_{LF} \cdot \left(\frac{h\bu\otimes h\bu}{h}\right)_E^{(i)}+{\nabla_W\Big(\bar H^{(i)}_E H^{(i)}_{I,2}+\frac12\eps^2 (H^{(i)}_{I,2})^2-H^{(i)}_{I,2}b\Big)+H^{(i)}_{I,2}\nabla_W b}\right).
	\end{equation}
\end{subequations}
\eit

\subsection{Well-balanced property for high order semi-implicit scheme}
Here we show that high order semi-implicit scheme \eqref{S3_UE2}-\eqref{S3_Unp1} can maintain the well-balanced property for
the still water equilibrium \eqref{still}. We have the following theorem:
\begin{thm}
	The high order semi-implicit scheme \eqref{S3_UE2}-\eqref{S3_Unp1} is well-balanced for the still water equilibrium, in the sence that, if initially
	the water is at still, namely $H^0=$ Const. and $(h\bu)^0 = {\bf 0}$, the scheme
	can maintain still water at any later time with $H^n=$ Const. and $(h\bu)^n={\bf 0}$.
\end{thm}

\begin{proof}

We prove this theorem using the mathematical induction. Assume at the time step $t^n$, we have $H^n=$ Const. and $\bu^n={\bf 0}$. First for $U^{(0)}_E=U^{(0)}_I=U^n$, similar to the discussion in the first order semi-implicit scheme, we have
\begin{equation}
\nabla_{LF}\cdot (h\bu)_I^n = 0, \quad \nabla_{LF} \cdot \left(\frac{h\bu\otimes h\bu}{h}\right)_E^n = 0, \quad {H^n_{I,2} = 0}, \quad \nabla ^2_C:\left(\frac{h\bu\otimes h\bu}{h}\right)^n = 0.
\end{equation}
By the induction hypothesis, we assume $H^{(j)}_E=H^{(j)}_I=$ Const. and $(h\bu)^{(j)}_E=(h\bu)^{(j)}_I={\bf 0}$ hold for any $j\le i-1$, from which we have
for $j=1,\cdots,i-1$:
\begin{equation}
\label{flux0}
\nabla_{LF}\cdot (h\bu)_I^{(j)} = 0, \quad \nabla_{LF} \cdot \left(\frac{h\bu\otimes h\bu}{h}\right)_E^{(j)} = 0, \quad {H^{(j)}_{I,2} = 0},\quad \nabla ^2_C:\left(\frac{h\bu\otimes h\bu}{h}\right)^{(j)}_E = 0.
\end{equation}
The goal is to show that $H^{(i)}_E=H^{(i)}_I=$ Const. and $(h\bu)^{(i)}_E=(h\bu)^{(i)}_I={\bf 0}$. From \eqref{S3_UE2}, we have $h^{(i)}_E=h^n$ so that $H^{(i)}_E=h^n+b=$ Const., and $(h\bu)^{(i)}_E=(h\bu)^n={\bf 0}$. Similarly $h_*^{(i)}+b=h^n+b=$ Const., and $(h\bu)^{(i)}_*=(h\bu)^n={\bf 0}$ from \eqref{S3_Ustar2}. With these, we conclude that $h^{**}=0$ from \eqref{S3_HI22}, and solving the elliptic equation \eqref{S3_HI21} with a positive definite matrix leads to $H^{(i)}_{I,2}=0$. Furthermore, due to $(h\bu)^{(i)}_E={\bf 0}$ and $H^{(i)}_{I,2}=0$, we get $(h\bu)^{(i)}_I=(h\bu)^{(i)}_*={\bf 0}$ from \eqref{S3_UI22}. It follows that $h^{(i)}_I=h^{(i)}_*$ from \eqref{S3_UI21}, so that $H^{(i)}_I=h^{(i)}_*+b=$ Const.,
and we complete the mathematical induction.

Since \eqref{flux0} holds for $j=1,\cdots,s$, substituting them into \eqref{S3_Unp1}, we obtain $H^{n+1}=H^n=$ Const. and $(h\bu)^{n+1}=(h\bu)^n={\bf 0}$. Therefore, the well-balanced property is preserved and this finishes the proof.
\end{proof}

\subsection{Asymptotic preserving and asymptotically accurate properties}

In this section, we formally prove the AP property for the first order semi-implicit scheme \eqref{S3_hu}, and the AA property for the high order semi-implicit scheme \eqref{S3_UE}-\eqref{S3_Un1}.
When we discuss the AP or AA property, we focus on the time discretization while keeping the space continuous. First we have the following theorem.
\begin{thm}
	The first order semi-implicit scheme \eqref{S3_hu} with space continuous is asymptotic preserving, in the sense that, with no-slip or periodic boundary condition, at the leading order asymptotic expansions, the scheme \eqref{S3_hu} is a consistent approximation of the lake equations \eqref{lake} at the zero Froude number limit.
\end{thm}

\begin{proof}
To prove the theorem, we assume the following expansions of the solutions at all time levels, i.e., $h^n(\bx):=h(\bx,t^n)$ and $\bu^n(\bx):=\bu(\bx,t^n)$ admit
\begin{equation}
h^n(\bx)=h^n_0(\bx)+\eps^2 H^n_2(\bx), \quad \bu^n(\bx)=\bu^n_0(\bx)+\eps \bu^n_1(\bx),
\end{equation}
and correspondingly the water surface level $H^n(\bx):=H(\bx,t^n)=h(\bx,t^n)+b(\bx)$ takes the form
\begin{equation}
H^n(\bx)=h^n_0(\bx)+b(\bx)+\eps^2 H^n_2(\bx),
\end{equation}
where $H_0=h^n_0(\bx)+b(\bx)=$ Const., namely $h^n_0(\bx)$ does not depend on $n$.

We plug them into the semi-discrete scheme \eqref{S3_hu}, with the first equation equivalent to the first equation of \eqref{S3_E1}. From $H^{n+1}=\bar{H}^n+\eps^2 H^{n+1}_2$, it yields $\bar{H}^n=H_0=$ Const.
Equating to zero for the $\mathcal{O}(\eps^0)$ terms, we have
\begin{equation}
\label{S3_1stap}
\left\{
\begin{aligned}
&\nabla \cdot (h_0\bu_0)^{n+1}   = 0, \\
&\frac{(h_0\bu_0)^{n+1} -(h_0\bu_0)^n}{\Delta t} + \nabla \cdot\left(\frac{h_0\bu_0\otimes h_0\bu_0}{h_0}\right)^n  + h_0^{n}\nabla H^{n+1}_2 =0,
\end{aligned}
\right.
\end{equation}
which is a consistent discretization to the lake equations \eqref{lake}, with $H^{n+1}_2$ solved from
\begin{equation}
-\nabla\cdot(h_0^n\nabla H^{n+1}_2) =\nabla ^2:\left(\frac{h_0\bu_0\otimes h_0\bu_0}{h_0}\right)^n.
\end{equation}
\end{proof}

Now we are ready to present the AA property for the high order semi-implicit scheme \eqref{S3_UE}-\eqref{S3_Un1}, that is, the scheme maintains its temporal order of accuracy for the lake equations \eqref{lake} at the zero Froude limit when $\eps \rightarrow 0$ \cite{Lorenzo2005,boscarino2021high}. To have the AA property, it is crucial that the IMEX Runge-Kutta scheme \eqref{DBT} is stiffly accurate (SA), namely the implicit part satisfies $b^T={\bf e}^T_s A$, where ${\bf e}^T_s=(0,\cdots,0,1)$ \cite{boscarino2013implicit,boscarino2021high}.
Besides, the initial conditions $(h^0(\bx),h^0(\bx)\bu^0(\bx))$ need to be well-prepared, in the sense that
\begin{equation}
\label{wellp}
h^0(\bx)=h_0(\bx)+\eps^2 H^0_2(\bx), \quad \bu^0(\bx)=\bu_0(\bx)+\eps \bu_1(\bx), \quad\text{ and }\quad h_0(\bx)+b(\bx)=H_0=\text{Const}.
\end{equation}
We have the following theorem about the AA property of the high order semi-implicit scheme.
\begin{thm}
	For the high order semi-implicit scheme \eqref{S3_UE}-\eqref{S3_Un1} of temporal order $p$, when applied to the system \eqref{S2_E1} on a bounded domain with no-slip or periodic boundary condition, suppose the IMEX Runge-Kutta scheme \eqref{DBT} is stiffly accurate, and the initial conditions $(h^0(\bx),h^0(\bx)\bu^0(\bx))$ are well prepared \eqref{wellp}. Denoting by
	${\bf V}^1({\bf x};\eps)=(h^1(\bx;\eps),h^1(\bx;\eps)\bu^1(\bx;\eps))$ the numerical solution after one time step, we have
    \begin{equation}\label{Prop1}
    \lim_{\eps \to 0} h^1(\bx;\eps)+b(\bx)=H_0, \quad \lim_{\eps \to 0} \nabla \cdot \Big(h^1(\bx;\eps)\bu^1({\bx};\eps)\Big)=0.
    \end{equation}
    Furthermore, let $\bV_{lake}(\bx,t)=(h_{lake}(\bx,t),h_{lake}(\bx,t)\bu_{lake}(\bx,t))$ be the exact solution of the lake equations \eqref{lake} with initial conditions $(h^0(\bx),h^0(\bx)\bu^0(\bx))$, one has the one-step error estimate
    \begin{equation}
    \label{Prop2}
    \lim_{\eps\to 0} {\bf V}^1({\bf x};\eps)={\bf V}_{lake}({\bf x}, \Delta t)+\mathcal{O}(\Delta t^{p + 1}),
    \end{equation}
    i.e., the high order semi-implicit scheme is AA.
\end{thm}
	The proof follows from the same structure as in \cite{boscarino2021high} by the mathematical induction, and is skipped here.

\section{Numerical tests}
\label{sec4}
\setcounter{equation}{0}
\setcounter{figure}{0}
\setcounter{table}{0}

In this section, we will perform some numerical tests with the Froude number ranging from $0$ to $\mO(1)$. The fifth order finite difference WENO reconstruction \cite{shu1998essentially,shu2009high,xing2005high} is used for the first order spatial derivatives, and the fourth order (compact) central difference discretizations for the second order derivatives. In time we employ a third order SA IMEX Runge-Kutta scheme SI-IMEX(4,4,3) from \cite{boscarino2021high}, with the double Butcher tableau given by
\begin{align} \label{IMEX1_(4,4,3)}
&\textrm{\bf Explicit :} \nonumber \\ \vspace{8mm}
&\begin{array}{c|cccc}
0 & 0 & 0 & 0 & 0 \\
\gamma & \gamma & 0 & 0 & 0\\
0.717933260754 & 1.243893189483& -0.525959928729 & 0 & 0\\
1 &   0.630412558153 & 0.786580740199 &  -0.416993298352& 0\\
\hline
0 & 0 & 1.208496649176& -0.644363170684 & \gamma
\end{array},\nonumber \\ \vspace{8mm}
&\textrm{\bf Implicit :}  \\ \vspace{8mm}
&\begin{array}{c|cccc}
\gamma&  \gamma & 0  & 0 & 0\\
\gamma & 0&  \gamma & 0 & 0\\
0.717933260754 &0 & 0.282066739245 & \gamma & 0\\
1 &0 & 1.208496649176& -0.644363170684 & \gamma\\
\hline
&0 & 1.208496649176& -0.644363170684 & \gamma
\end{array},\nonumber
\end{align}
where $\gamma = 0.435866521508$.

The time step is taken as
\[
\Delta t = \text{CFL}\,\Delta x/\Lambda, \qquad \Lambda = \max \{|\bu|+\min(1,1/\eps)\sqrt{h}\},
\]
and $\text{CFL}=0.2$ is used.
For the accuracy tests, the time step is modified to
$\Delta t = \text{CFL}\Delta x^{5/3}/\Lambda$, for better observation of spatial orders. $N$ or $N^2$ uniform gird points are used for 1D and 2D problems respectively, except otherwise specified.

For the Froude number $\eps$ of $\mO(1)$, {e.g. $\eps=\frac{1}{\sqrt{g}}$ and $g=9.812$ is the gravitational constant}, we will compare our results to reference solutions, which are produced by the fifth order well-balanced finite difference WENO scheme developed by Xing and Shu \cite{xing2005high}. We refer it as ``WB-Xing" in the following.

\subsection{One dimensional case}

\begin{exa} {\em
\label{exam1}
({\bf{Accuracy test}})
We first consider an example with smooth initial conditions and a non-zero bottom topology, which are given by \cite{xing2005high}
\begin{equation}
\label{Ex1_1}
h(x,0)  = 5+ \exp(\cos(2\pi x)),\quad
(hu)(x,0) = \sin(\cos(2\pi x)), \quad
b(x)    = \sin^2(\pi x),      \quad
x \in [0,1].
\end{equation}
Periodic boundary condition is used, with the Froude number $\eps=\frac{1}{\sqrt{g}}$.
We take mesh grid points $N=40\times2^i$ for $i=1,\cdots,5$. Since the exact solution is not available, numerical errors are computed by comparing the numerical solutions at two successive mesh grid points from refinement. For this case, the errors are computed for the momentum $hu$, at a final time $T=0.1$. The results are shown in Table \ref{T_eg_1}, we can see that it approaches fifth order,
which is similar to the results in \cite{xing2005high}.
\begin{table}[htbp]
\caption{Example~\ref{exam1}: numerical errors and orders of accuracy for the momentum $hu$. $T=0.1$.}
\begin{center}
\begin{tabular}{c|c|c|c|c|c}\hline
  N           &    80    &    160   &   320    &  640     &  1280     \\ \hline
  $L_1$ error & 3.35E-02 & 4.61E-03 & 4.44E-04 & 2.06E-05 & 6.99E-07 \\ \hline
   order      &    --    &   2.86   &   3.37   &   4.43   &   4.88   \\ \hline
\end{tabular}
\end{center}
\label{T_eg_1}
\end{table}
}
\end{exa}

\begin{exa}{\em
\label{exam2}
({\bf{Accuracy test for a range of $\eps$}})
In this example, we try to test the orders of accuracy for our scheme in different regimes of the Froude number. We take the non-flat bottom topological function  $b(x)$ as
$$
b(x) = 1+\sin(2\pi x),	
$$
with initial conditions
\begin{equation}
\label{ini2}
\begin{aligned}
&h(x,0)     = 10 - b(x) + \eps^2\exp(\sin(2\pi x)), \\
&(hu)(x,0)  = 1 + \eps^2\sin(2\pi x).
\end{aligned}
\end{equation}
Similarly, periodic boundary condition is used.
Numerical errors are computed in the same way as in the previous example.
Three different Froude numbers $\eps =1,10^{-2}, 10^{-4}$ are taken, with a final time $T=0.05$. And the computational domain is $x\in [0,2]$.
Numerical errors and orders are shown in Table {\ref{T_eg2}}.
From this table, we can see that for all $\eps$'s, almost fifth order accuracy can be observed,
which indicate that our scheme is asymptotically accurate.

\renewcommand{\multirowsetup}{\centering}
\begin{table}[htbp]
\caption{ Example~\ref{exam2}: numerical errors and orders of accuracy for the momentum $hu$  with initial conditions \eqref{ini2}. $T=0.05$.}
\begin{center}
\begin{tabular}{c||c|c|c|c|c}\hline\hline
  $\eps$  &      N       &     80   &   160    &   320    &   640                 \\ \hline\hline
  \multirow{2}{1cm}{1}
  &  $L^1$ error & 6.09E-03   & 3.23E-04 & 1.16E-05 & 4.05E-07     \\ \cline{2-6}
          &  order &    --    &   4.24   &   4.80   &   4.84          \\ \hline\hline
  \multirow{2}{1cm}{$10^{-2}$}
  &  $L^1$ error   & 8.27E-03 & 3.75E-04 & 2.70E-05 & 1.06E-06     \\ \cline{2-6}
          &  order &    --    &   4.46   &   3.79   &   4.67       \\ \hline\hline
  \multirow{2}{1cm}{$10^{-4}$}
  &  $L^1$ error   & 4.58E-05 & 4.92E-06 & 1.26E-06 & 5.58E-08     \\ \cline{2-6}
          &  order &    --    &   3.22   &   1.97   &   4.49           \\ \hline\hline

\end{tabular}
\end{center}
\label{T_eg2}
\end{table}	
}\end{exa}

\begin{exa}{\em
\label{exam3}
({\bf{A small perturbation of a steady-state water}})
This example was first proposed by LeVeque in \cite{leveque1998balancing} and later studied by Xing and Shu in \cite{xing2005high}. There is a small perturbation on a quasi-stationary water, moving over a non-flat bottom topography. The bottom function is smooth, which is given by
\begin{equation}
\label{Ex1}
b(x) =
\left\{
\begin{aligned}
& 0.25(\cos(10\pi(x-1.5))+1),   &  \text{if } 1.4\le x\le 1.6; \\
& 0,                            &  \text{otherwise},
\end{aligned}
\right.
\end{equation}
and the initial conditions are:
\begin{subequations}
	\label{ini_ex3}
\begin{equation}
		\label{Ex3_1}
		h(x,0)=
		\left\{
		\begin{aligned}
		&1-b(x)+\eta,       & \text{if }  1.1\le x \le 1.2;\\
		&1-b(x),            & \text{otherwise},
		\end{aligned}
		\right.
\end{equation}
\begin{equation}
\label{Ex3_2}
hu(x,0) = 0,
\end{equation}
\end{subequations}
on the domain $x\in [0,2]$, see Fig. \ref{Fg_eg3_1}. $\eta$ is the magnitude of perturbation. Two cases are considered: $\eta=0.2$ (big pulse) and $\eta=0.001$ (small pulse). The Froude number is taken to be $\eps = \frac{1}{\sqrt{g}}$. After the perturbation moves over the non-flat bottom, two disturbances will generate and one propagates to the left and the other to the right, both with a speed $\sqrt{gh}$.

This example is used to test the well-balanced property of the numerical scheme.
For non well-balanced schemes, numerical errors may pollute the small perturbations.
We show the water surface level $H$ and momentum $hu$, for $\eta=0.2$ and $\eta=0.001$ in Fig.~\ref{Fg_eg3_2} and Fig.~\ref{Fg_eg3_3} respectively, at a final time $T=0.2$ with $N=200$. We compare the solutions to the reference solutions of ``WB-Xing" with $N=3000$. It can be observed that, for both cases, our solutions can well capture the disturbances and match the reference solutions.

\begin{figure}[hbtp]
\begin{center}
	\mbox{\subfigure[$\eta=0.2$]
		{\includegraphics[width=7cm]{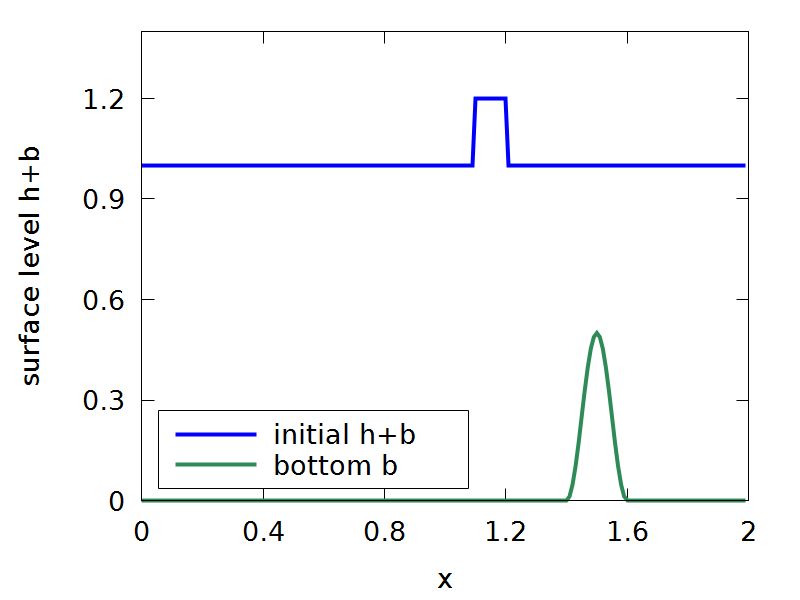}}\quad
		\subfigure[$\eta=0.001$]
		{\includegraphics[width=7cm]{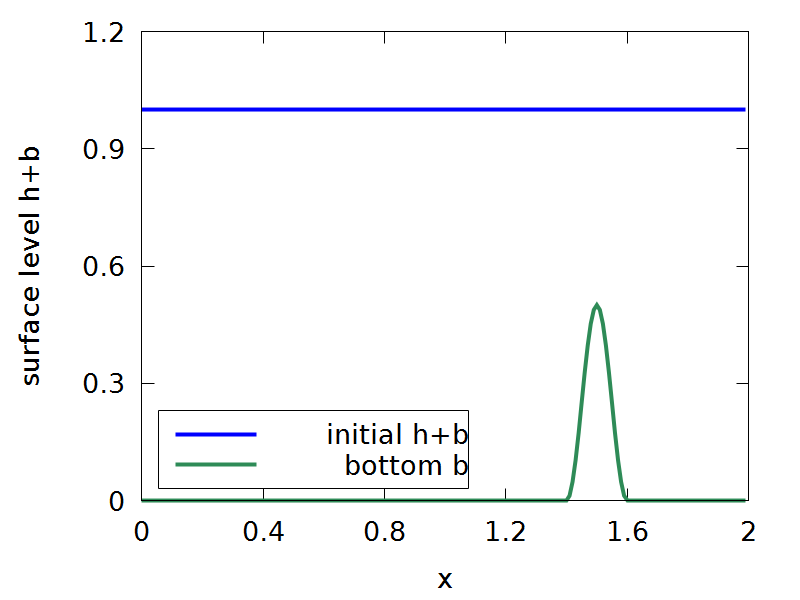}}
		}
\caption{ The initial water surface level $h+b$ \eqref{Ex3_1} and the bottom $b$ \eqref{Ex1} for Example~\ref{exam3}. Left: $\eta =0.2$; Right:$\eta=0.001$. }
\label{Fg_eg3_1}
\end{center}
\end{figure}
		
\begin{figure}[hbtp]
\begin{center}
\mbox{
	{\includegraphics[width=8cm]{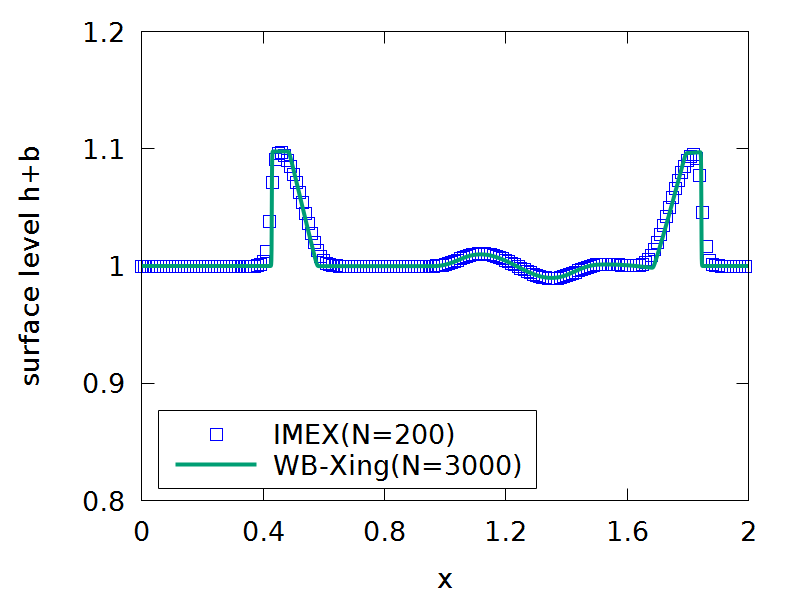}}\quad
					
	{\includegraphics[width=8cm]{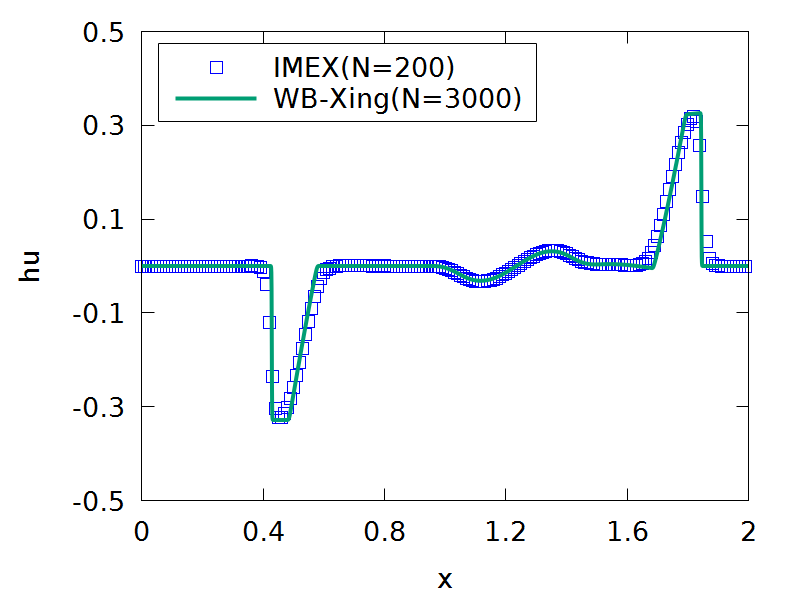}}
}
\caption{The water surface level $h+b$ (left) and momentum $hu$ (right) at time $T=0.2$ with $\eta=0.2$ for Example \ref{exam3}. }
\label{Fg_eg3_2}
\end{center}
\end{figure}
		
\begin{figure}[hbtp]
\begin{center}
\mbox{
	{\includegraphics[width=8cm]{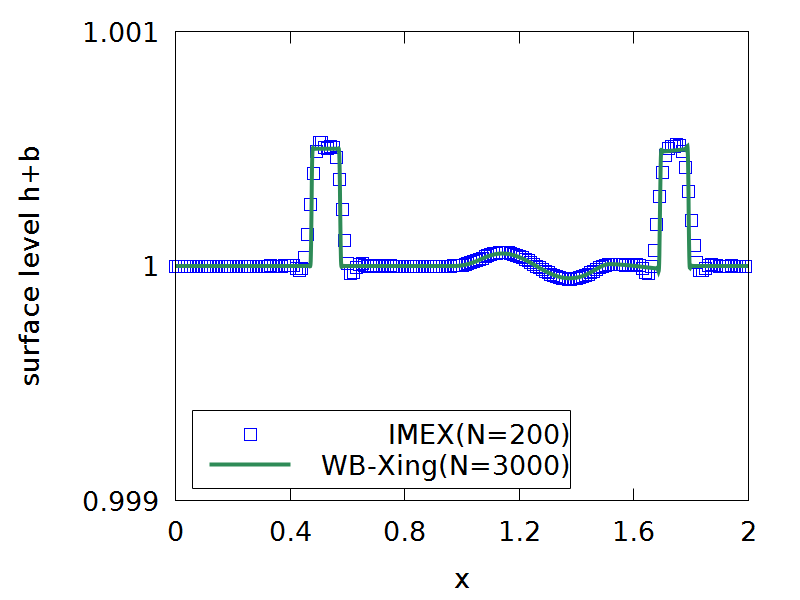}}\quad					
	{\includegraphics[width=8cm]{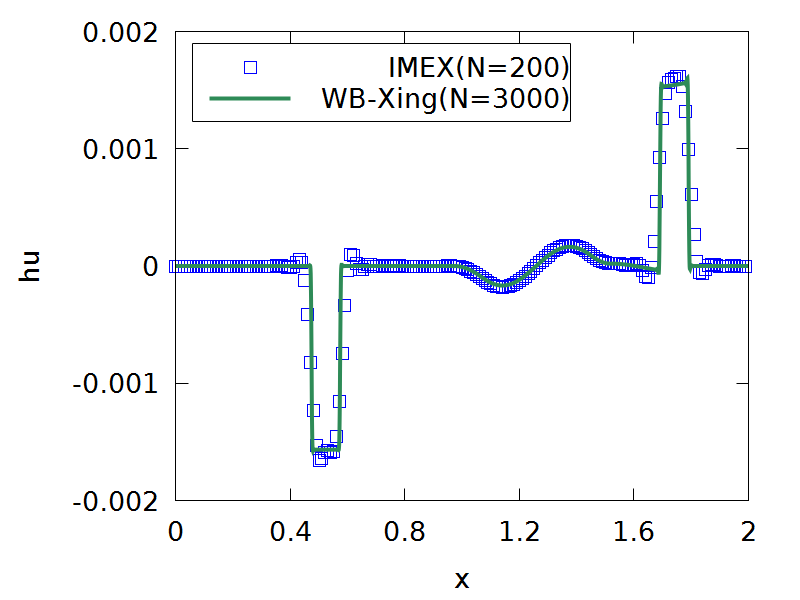}}
}
\caption{The water surface level $h+b$ (left) and momentum $hu$ (right) at time $T=0.2$ with $\eta=0.001$ for Example \ref{exam3}.  }
\label{Fg_eg3_3}
\end{center}
\end{figure}
}\end{exa}

\begin{exa}{\em
\label{exam4}
({\bf{Dam breaking}}) The dam breaking problem over a rectangular bump is widely used to test the oscillation-free property of numerical schemes for the shallow water equations \cite{vukovic2002eno}. The bottom function is defined as
\begin{equation}
		\label{Ex3}
		b(x) =
		\left\{
		\begin{aligned}
		&8,   & \text{if } |x-750| \le 1500/8;\\
		&0,   & \text{otherwise},
		\end{aligned}
		\right.
\end{equation}
and the initial conditions are:
\begin{equation}
		\label{Ex4}
		(hu)(x,0) = 0\quad
		\text{ and } \quad
		h(x,0)=
		\left\{
		\begin{aligned}
		&20-b(x),  & \text{if } x\le 750; \\
		&15-b(x),  & \text{otherwise},
		\end{aligned}
		\right.
\end{equation}
on a computational domain $x\in[0,1500]$. The Froude number is taken as $\eps =\frac{1}{\sqrt{g}}$. The inflow and outflow boundary conditions are set the same as the initial values on the left and right respectively.
As time evolves, the initial jump on $h$ will generate two waves.
One is a rarefaction wave traveling to the left, and the other is a shock traveling to the right. For this example, we show the water surface level $h+b$ on the mesh points $N=500$ at two different times $T=15$ and $60$ in Fig. {\ref{Fg_eg4_1}} and Fig. {\ref{Fg_eg4_2}}, respectively. We also compare them to the reference solutions on $N=3000$ with ``WB-Xing" method. The results match each other well.

		\begin{figure}[hbtp]
			\begin{center}
				\mbox{
					{\includegraphics[width=8cm]{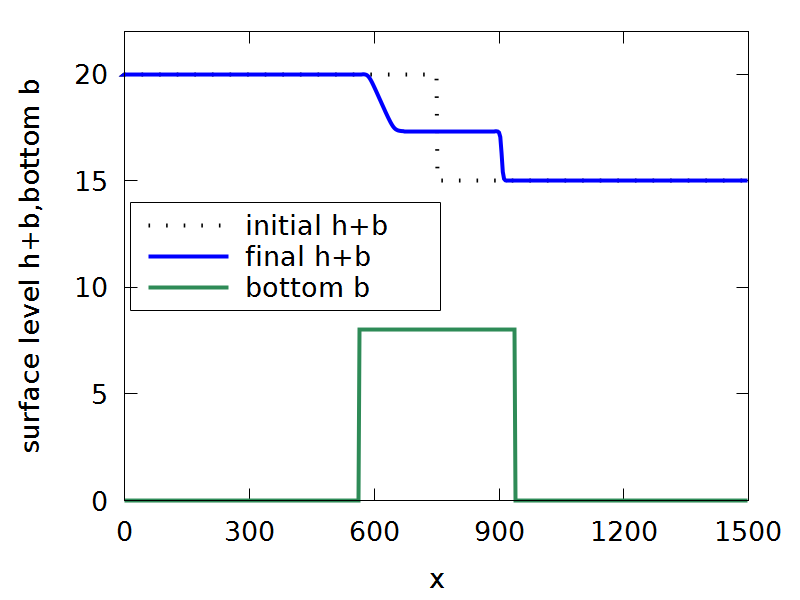}}\quad
					
					{\includegraphics[width=8cm]{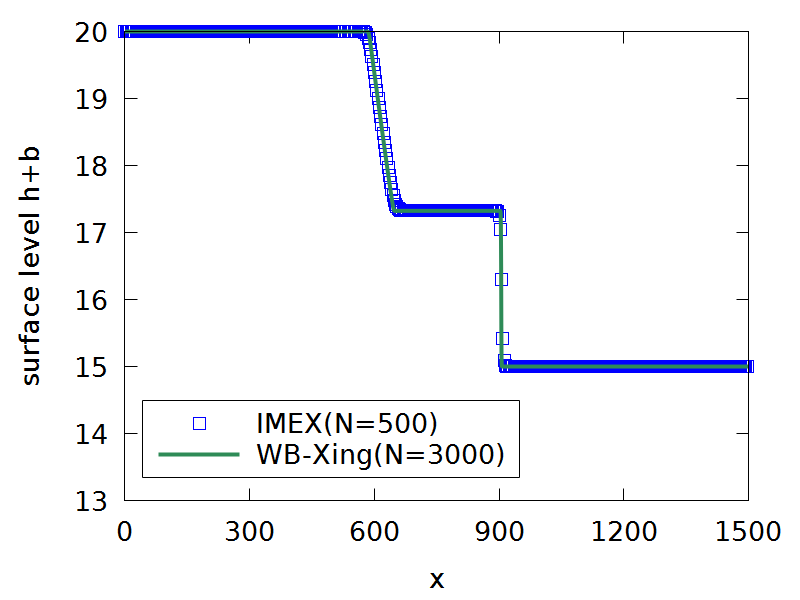}}
				}
				\caption{The water surface level for the dam breaking problem of Example \ref{exam4} at $T=15$. Left: initial surface level $h+b$, final surface level $h+b$ and bottom topology $b$; Right: comparisons to the reference solutions.}
				\label{Fg_eg4_1}
			\end{center}
		\end{figure}
		
		\begin{figure}[hbtp]
			\begin{center}
				\mbox{
					{\includegraphics[width=8cm]{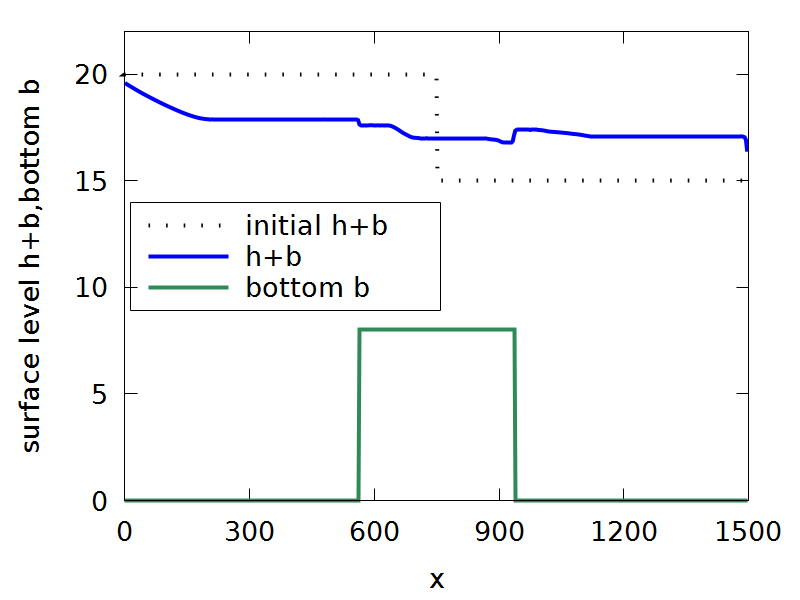}}\quad
					
					{\includegraphics[width=8cm]{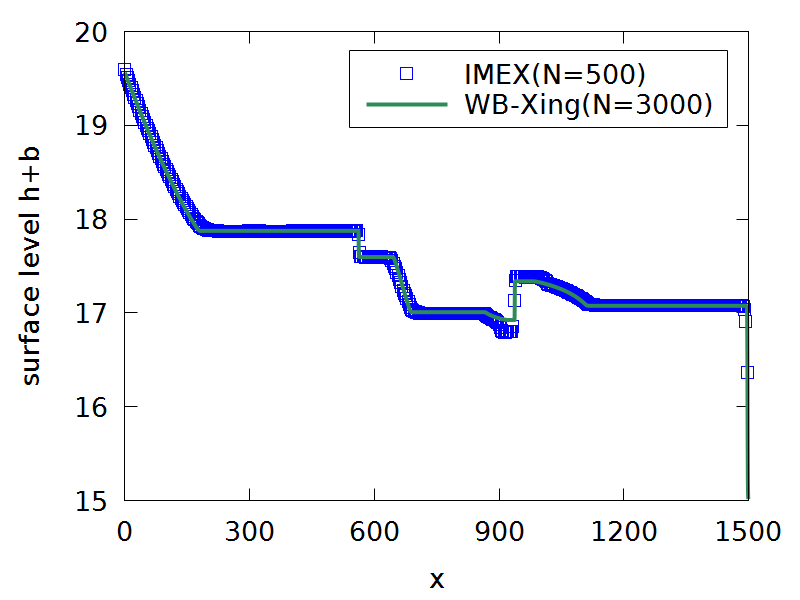}}
				}
				\caption{The water surface level for the dam breaking problem of Example \ref{exam4} at $T=60$. Left: initial surface level $h+b$, final surface level $h+b$ and bottom topology $b$; Right: comparison  to the reference solutions.}
				\label{Fg_eg4_2}
			\end{center}
		\end{figure}

}\end{exa}

\begin{exa}{\em
\label{exam5}
({\bf{Lake at rest}}) In this example,  we consider a still water initially to test the well-balanced property of our scheme.  We take a non-smooth bottom topology given by
%\begin{equation}
%		\label{Ex5}
%		b(x) = 5e^{-\frac{2}{5}(x-5)^2}
%\end{equation}
\begin{equation}
		\label{Ex6}
		b(x) =
		\left\{
		\begin{aligned}
		&4,  &\text{if } 4\le x \le 8; \\
		&0,  &\text{otherwise},
		\end{aligned}
		\right.
\end{equation}
and the initial conditions are
\begin{equation}
		\label{Ex7}
		(h+b)(x,0) = 10,\quad  (hu)(x,0) =0,
\end{equation}
on a computational domain $[0,10]$ with the Froude number set to be $\eps =\frac{1}{\sqrt{g}}$. Periodic boundary condition is adopted.

In Fig.~{\ref{Fg_eg5_2}}, {we show the variation of the water surface level, and the momentum at final time $T=10$.} We can clearly see the errors are within machine precision, namely, the still water equilibrium is well preserved.

To further show the ability of our scheme, we set the initial velocity as $u(x,0) =1$ for a moving water. Due to the non-flat bottom $b$, the initial water equilibrium will be destroyed. We compute the numerical solution to $T=0.1$. The results are shown in Fig.~{\ref{Fg_eg5_4}} and compared to the reference solutions from ``WB-Xing". We can see the results still match each other, and the discontinuities are well captured without any artificial oscillation.

%\begin{figure}[hbtp]
%	\begin{center}
%				\mbox{
%					{\includegraphics[width=8cm]{pic/eg4_1_1}}\quad
%					{\includegraphics[width=8cm]{pic/eg4_1_2}}
%				}
%				\caption{ The numerical solution for the lake at rest  with smooth bottom. Left :error of $h$ and $hu$; right: initial surface level $h+b$ and bottom topology $b$, the numerical result.}
%				\label{Fg_eg4_1}
%	\end{center}
%\end{figure}
		
		\begin{figure}[hbtp]
			\begin{center}
				\mbox{
					{\includegraphics[width=8cm]{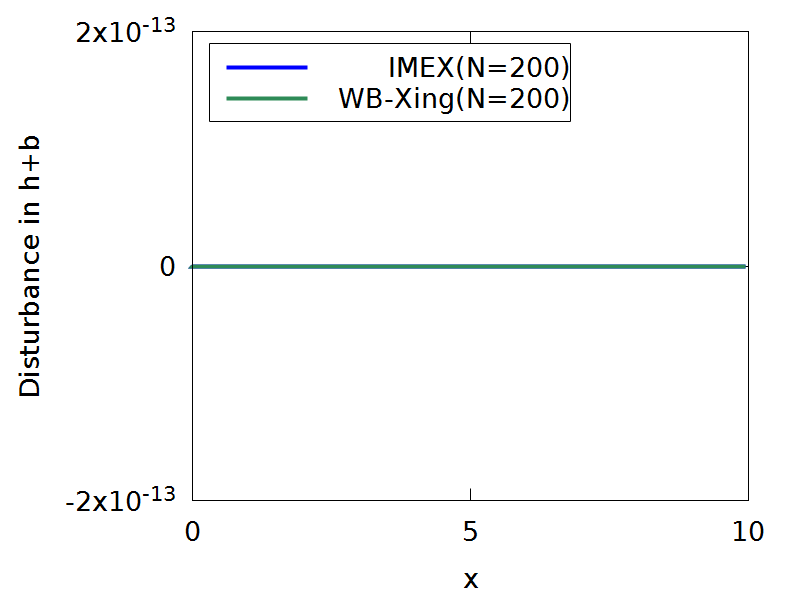}}\quad
					
					{\includegraphics[width=8cm]{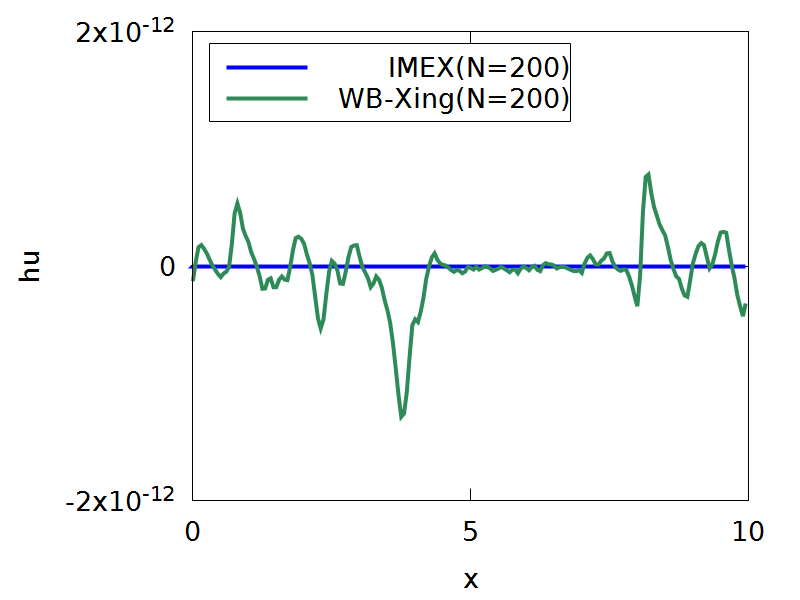}}
				}
				\caption{The lake at rest with non-smooth bottom for Example \ref{exam5} at $T=10$. Left: the disturbance of $H$; Right: the numerical result for momentum $hu$.}
				\label{Fg_eg5_2}
			\end{center}
		\end{figure}

		\begin{figure}[hbtp]
			\begin{center}
				\mbox{
					{\includegraphics[width=8cm]{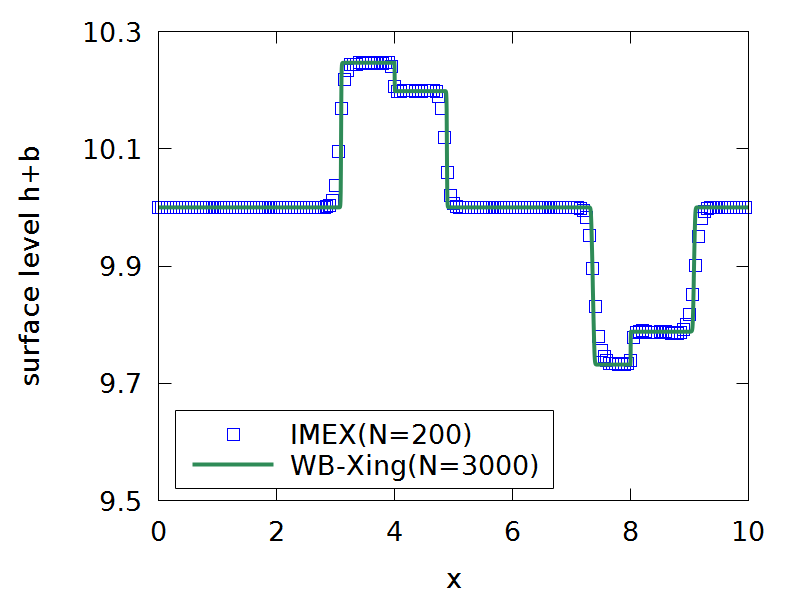}}\quad
					
					{\includegraphics[width=8cm]{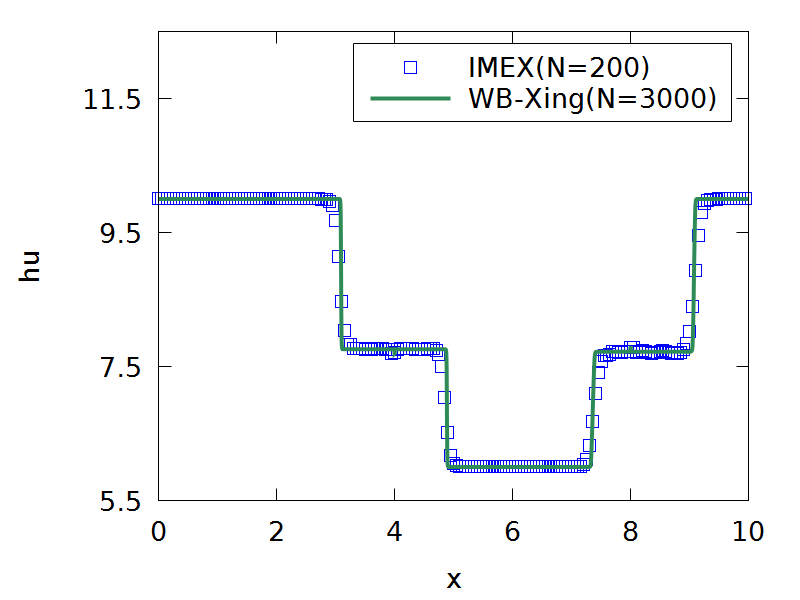}}
				}
				\caption{Moving water with non-smooth bottom for Example \ref{exam5} at $T=0.1$. Left: the water surface level $h+b$; Right: the momentum $hu$.}
				\label{Fg_eg5_4}
			\end{center}
		\end{figure}
		
}\end{exa}

\begin{exa}{\em
  \label{exam9}
In this example, we would like to test our scheme for a multiscale wave, which was  studied in \cite{klein1995semi, liu2020well}. The initial conditions are taken as
\begin{subequations}
	\label{Ex9}
           \begin{equation}
           \label{Ex9_1}
           H(x,0) = 1 + \frac{\sigma(x)}{2} \sin(\eps40\pi x) +
                    \eps(1+\cos(\eps\pi x)),
           \end{equation}
           \begin{equation}
           \label{Ex9_2}
           u(x,0) = \sqrt{2}(1+\cos(\eps\pi x)), \quad \quad b(x) = 0,
           \end{equation}
     \end{subequations}
  with the Froude number $\eps = 0.02$, and
           $$
           \sigma(x) = \left\{
           \begin{aligned}
		   &0.5(1-\cos(0.1\pi x)),  &\text{if } 0\le x \le 20; \\
		   &0,  &\text{otherwise}.
		   \end{aligned}
           \right.
           $$
The computational domain is $[-51,51]$ with periodic boundary condition.

%We take mesh grid points $N=1020$ and $N=2040$,
%\BL{We show the numerical solutions of the multiscale wave propagation in Fig.~{\ref{Fg_eg_91}} and Fig.~{\ref{Fg_eg_92}} with two uniform meshes of $N=1020$ and $N=2040$, respectively.
% %Fig.~{\ref{Fg_eg9_1}} and Fig.~{\ref{Fg_eg9_2}}.
%We compare our results to those produced by the explicit scheme of ``WB-Xing'' on the same mesh. We can see
%that our AP scheme can better preserve the amptitude of these waves than the explicit scheme, which is similar to the results in \cite{liu2020well}.}
{We show the numerical solutions of the multiscale wave propagation in Fig.~{\ref{Fg_eg_92}} on a uniform mesh of $N=2040$. We compare our results to those produced by the explicit scheme of ``WB-Xing'' on the same mesh. We can see
that under this mesh size, the results of our AP scheme match those from the explicit ``WB-Xing'' scheme very well.}
%           \begin{figure}[hbtp]
%			\begin{center}
%				\mbox{
%                     \subfigure[$t=0$]
%					{\includegraphics[width=8cm]{pic/eg9_01}}\quad
%					
%                     \subfigure[$t=0.2$]
%					{\includegraphics[width=8cm]{pic/eg9_11}}
%				}
%              \mbox{
%                     \subfigure[$t=0.5$]
%					{\includegraphics[width=8cm]{pic/eg9_21}}\quad
%					
%                     \subfigure[$t=1.0$]
%					{\includegraphics[width=8cm]{pic/eg9_31}}
%				}
%              \mbox{
%                     \subfigure[$t=2.4$]
%					{\includegraphics[width=8cm]{pic/eg9_41}}\quad
%					
%                     \subfigure[$t=4.1$]
%					{\includegraphics[width=8cm]{pic/eg9_51}}
%				}
%				\caption{ The numerical solution of water surface level $h+b$ for Example \ref{exam9} with a uniform mesh of $N=1020$.}
%            \label{Fg_eg_91}
%			\end{center}
%		\end{figure}

\begin{figure}[hbtp]
			\begin{center}
				\mbox{
                     \subfigure[$t=0$]
					{\includegraphics[width=8cm]{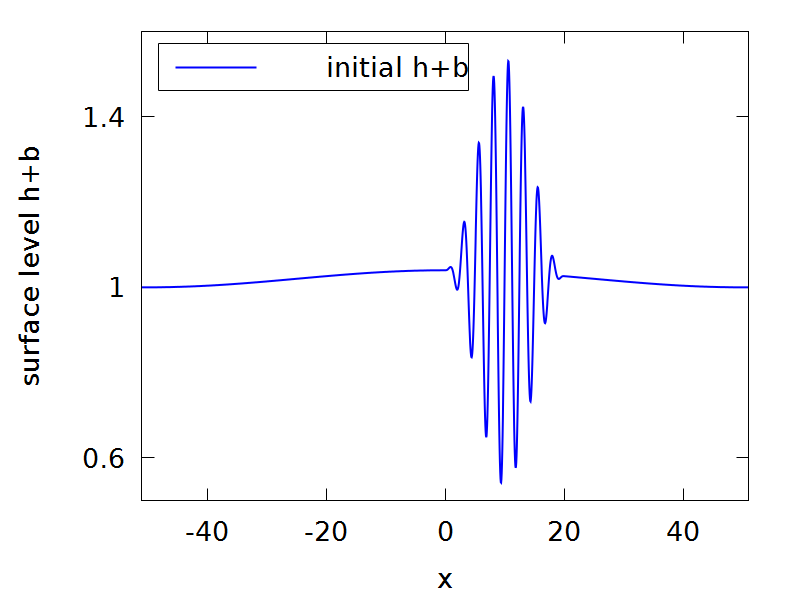}}\quad
					
                     \subfigure[$t=0.2$]
					{\includegraphics[width=8cm]{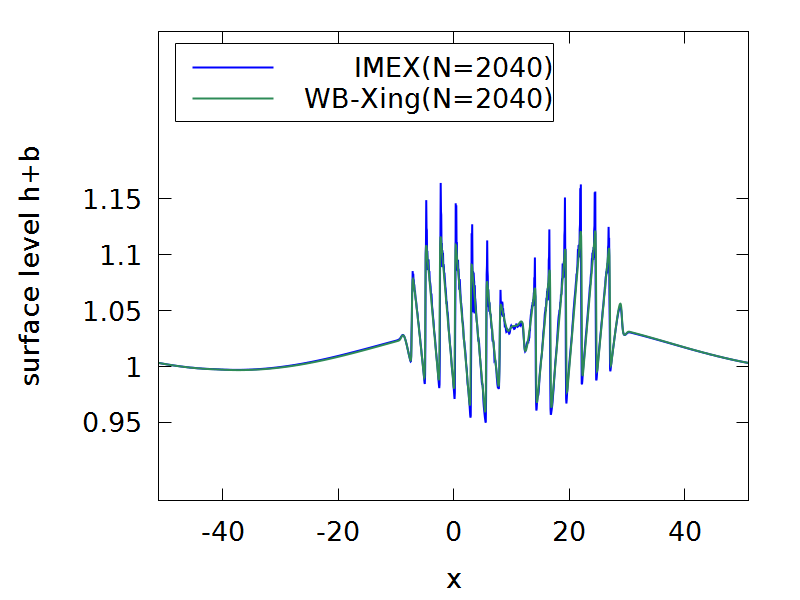}}
				}
              \mbox{
                     \subfigure[$t=0.5$]
					{\includegraphics[width=8cm]{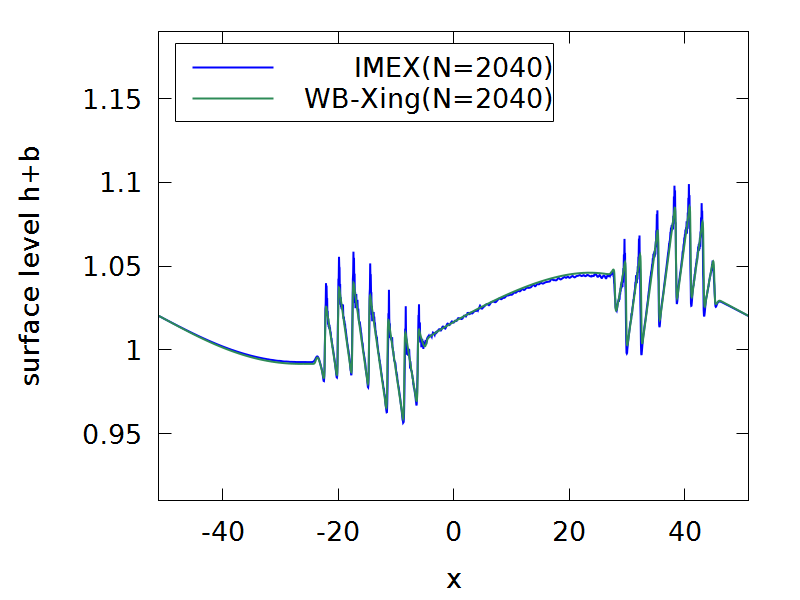}}\quad
					
                     \subfigure[$t=1.0$]
					{\includegraphics[width=8cm]{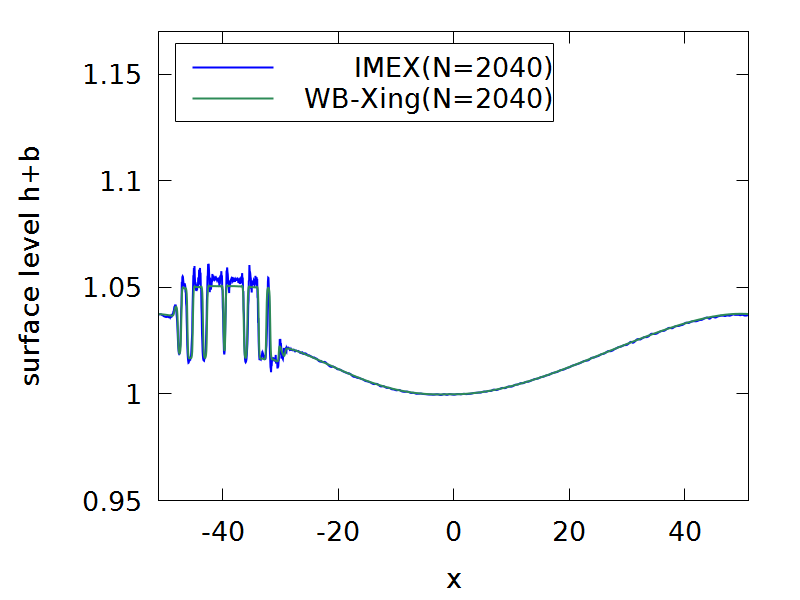}}
				}
              \mbox{
                     \subfigure[$t=2.4$]
					{\includegraphics[width=8cm]{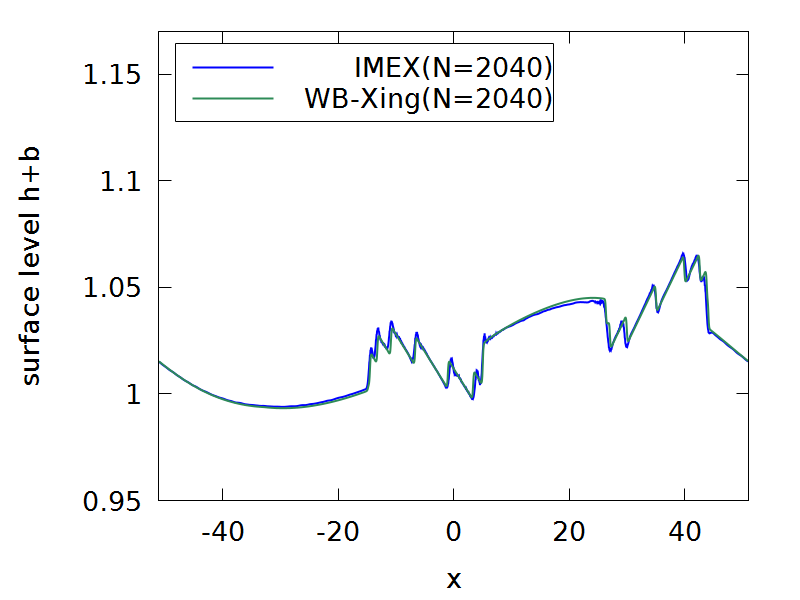}}\quad
					
                     \subfigure[$t=4.1$]
					{\includegraphics[width=8cm]{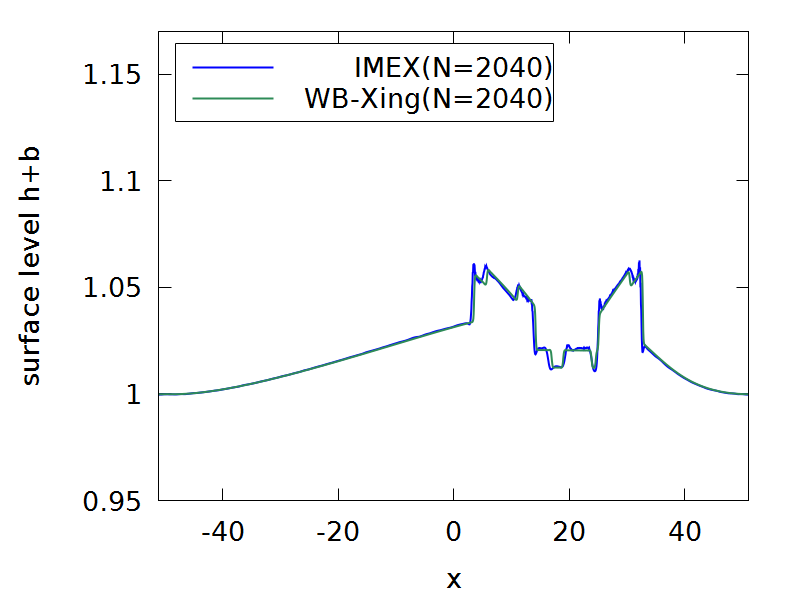}}
				}
				\caption{ The numerical solution of water surface level $h+b$ for Example \ref{exam9} with a uniform mesh of $N=2040$.}
            \label{Fg_eg_92}
			\end{center}
		\end{figure}

}\end{exa}

\subsection{Two-dimensional case}
\label{sec42}

\begin{exa}{\em
		\label{exam6}
		({\bf{Accuracy test}})
For this 2D example, we consider a smooth non-flat bottom function to be
\begin{equation}
		\label{E23}
		b(x,y) = \sin(2\pi x) + \cos(2\pi y) + 2,
\end{equation}
and the initial conditions are
\begin{equation}
		\label{E24}\left\{
		\begin{aligned}
		h(x,y,0)    & = 10 - b(x,y) + \eps^2\sin(2\pi x)\cos(2\pi y),  \\
		(hu)(x,y,0) & = \sin(2\pi x)\cos(2\pi y), \\
		(hv)(x,y,0) & = -\cos(2\pi x)\sin(2\pi y),
		\end{aligned}
		\right.
\end{equation}
on a computational domain $[0,1]^2$ with periodic boundary conditions in both directions. Note that the initial conditions \eqref{E24} are set to be well-prepared \eqref{wellp}.
		
We take three different $\eps$'s of $\eps=1, 10^{-2}$ and $10^{-4}$.  We compute the solution up to a final time $T=0.05$ on mesh grid points of $N^2$.
Since the exact solution is not available, the numerical errors are computed by comparing the solutions on two consecutive meshes.
The  $L^1$ errors and orders of accuracy are shown in Table~{\ref{T_eg6}}.
We can see high order accuracy can be obtained for all Froude numbers. 

\renewcommand{\multirowsetup}{\centering}
\begin{table}[htbp]
  \caption{ Example \ref{exam6}. Numerical errors and orders of accuracy for the momentum $hu$ with Froude number $\eps=1, 10^{-2}, 10^{-4} $.}
  \begin{center}
  \begin{tabular}{c||c|c|c|c|c|c}\hline\hline
  $\eps$ &  N         &    16    &    32    &    64    &    128   &   256   \\ \hline\hline
  \multirow{2}{1cm}{1}
  & $L^1$ error  & 6.56E-02 & 3.74E-03 & 1.36E-04 & 4.39E-06 & 1.39E-07 \\ \cline{2-7}
  & order &    --           &   4.13   &   4.78   &    4.96  &   4.98  \\ \hline\hline
  \multirow{2}{1cm}{$10^{-2}$}
  & $L^1$ error & 2.47E-02 & 1.46E-03 & 1.69E-04 & 6.11E-06 & 1.91E-07 \\ \cline{2-7}
  & order      &    --     &   4.08   &   3.11   &    4.79  &   5.00  \\ \hline\hline
  \multirow{2}{1cm}{$10^{-4}$}
  & $L^1$ error & 2.63E-02 & 1.33E-03 & 4.88E-05 & 1.63E-06 & 1.09E-07 \\ \cline{2-7}
  & order       &    --    &   4.31   &   4.76   &    4.90  &   3.91  \\ \hline\hline
  \end{tabular}
  \end{center}
  \label{T_eg6}
\end{table}
}\end{exa}

\begin{exa}
{\em
\label{exam8}
({\bf {A small perturbation of 2D steady-state water}})
For this example, we try to test our scheme for the capability of capturing the perturbation on a stationary water in the two dimensional case, which has been studied in \cite{leveque1998balancing,xing2005high}.

The bottom topography is an isolated elliptical shaped hump
    \begin{equation}
    \label{Ex8_1}
    b(x,y) = 0.8e^{-5(x-0.9)^2-50(y-0.5)^2}
    \end{equation}
and the initial conditions are
\begin{subequations}
    \begin{equation}
    \label{Ex8_2}
    h(x,y,0) = \left\{
    \begin{aligned}
    & 1 - b(x,y)+0.01,  &\text{if}\quad 0.05  \le x \le 0.15; \\
    & 1 - b(x,y),       &\text{otherwise}.
    \end{aligned}
    \right.
    \end{equation}
    \begin{equation}
    \label{Ex8_3}
    hu = hv = 0,
    \end{equation}
\end{subequations}
on the computational domain $[0,2]\times[0,1]$, with outflow boundary in the $x$ direction and periodic boundary in the $y$ direction. The Froude number is set as $\eps=\frac{1}{\sqrt{g}}$. We show the numerical results of surface level $H=h+b$ on two different meshes $200\times100$ and $400\times200$ in Fig.~{\ref{Fg_eg8_1}}. The initial perturbation is separated into two wave propagating to the left and right. With the left-propagating wave moving out of the domain, the right-propagating wave interacts with the non-flat bottom topography, and is well captured by the proposed method. We can observe that the numerical results are comparable to those of ``WB-Xing'' method in \cite{xing2005high}.

\begin{figure}[hbtp]
			\begin{center}
\mbox{\subfigure[surface level at $t=0.12$]					{\includegraphics[width=7cm]{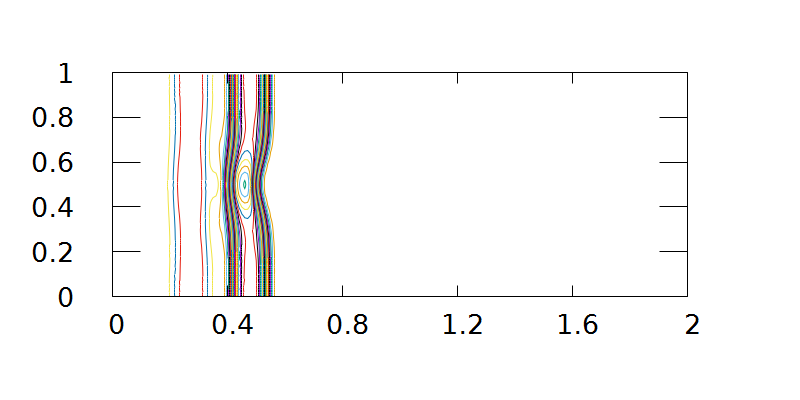}}\quad		
                     \subfigure[surface level at $t=0.12$]		
					{\includegraphics[width=7cm]{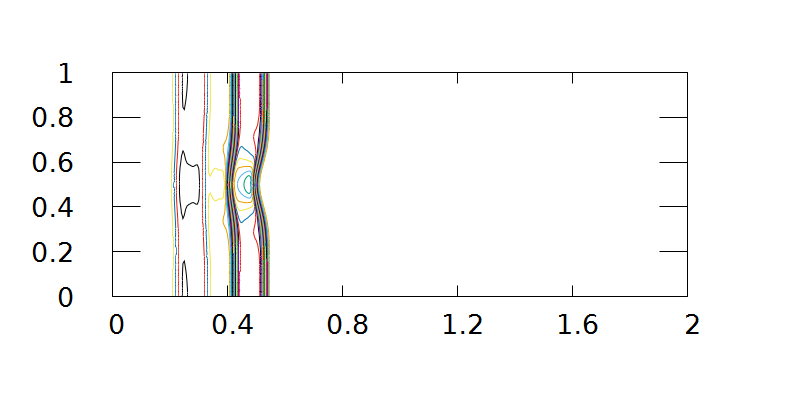}}
				}
\mbox{\subfigure[surface level at $t=0.24$]					{\includegraphics[width=7cm]{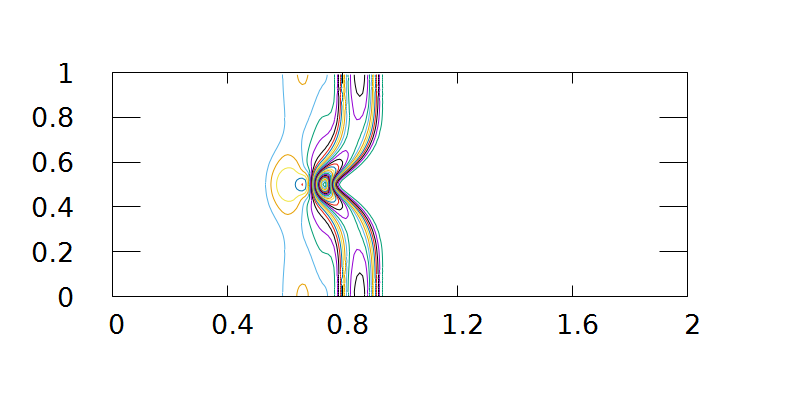}}\quad		
                     \subfigure[surface level at $t=0.24$]		
					{\includegraphics[width=7cm]{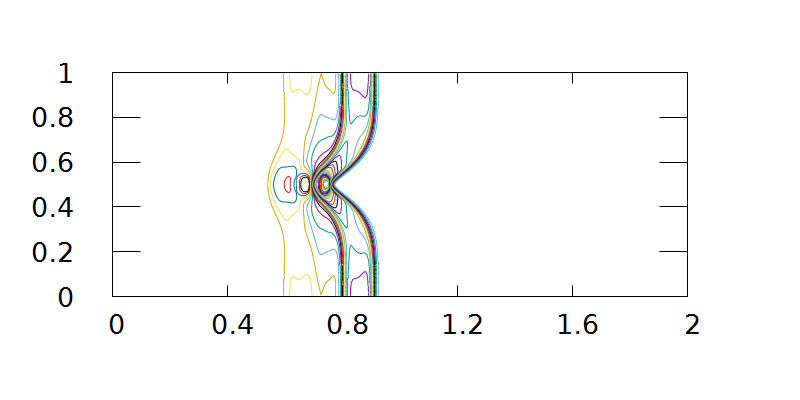}}
				}
\mbox{\subfigure[surface level at $t=0.36$]					{\includegraphics[width=7cm]{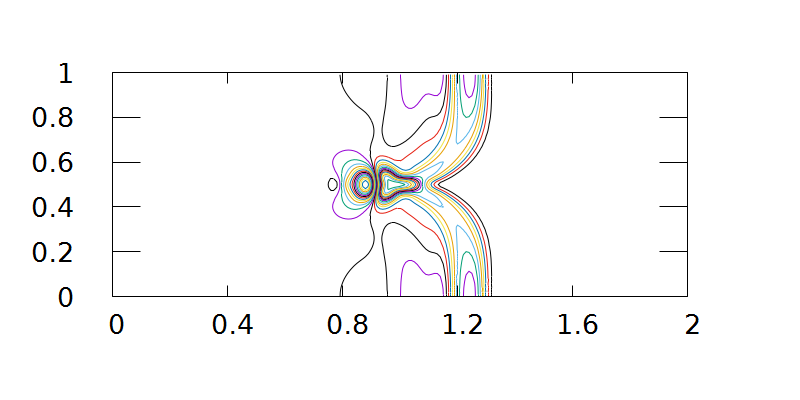}}\quad		
                     \subfigure[surface level at $t=0.36$]		
					{\includegraphics[width=7cm]{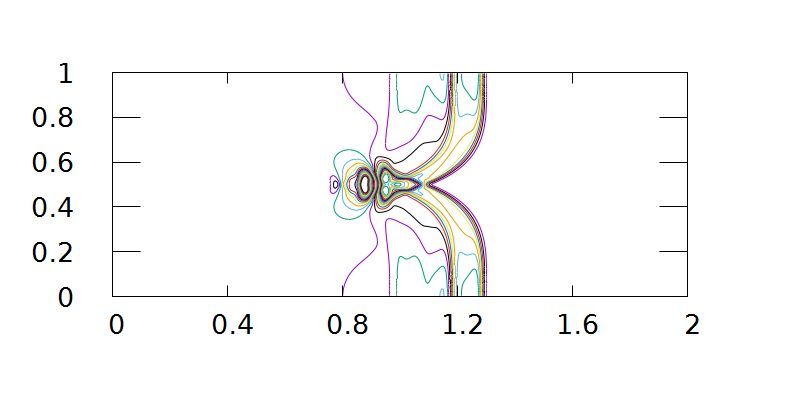}}
				}
\mbox{\subfigure[surface level at $t=0.48$]					{\includegraphics[width=7cm]{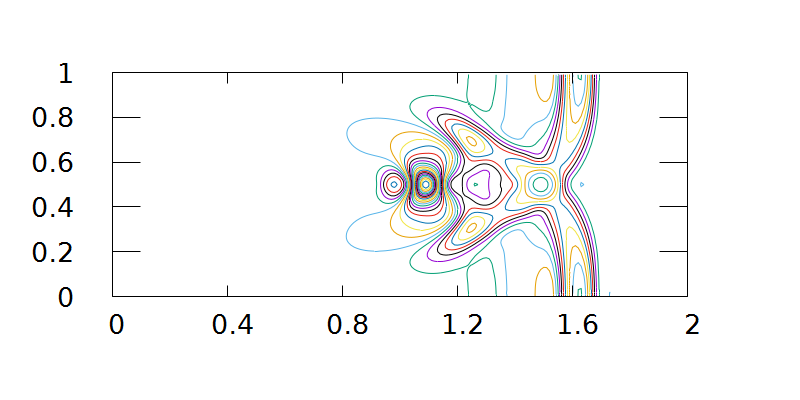}}\quad		
                     \subfigure[surface level at $t=0.48$]		
					{\includegraphics[width=7cm]{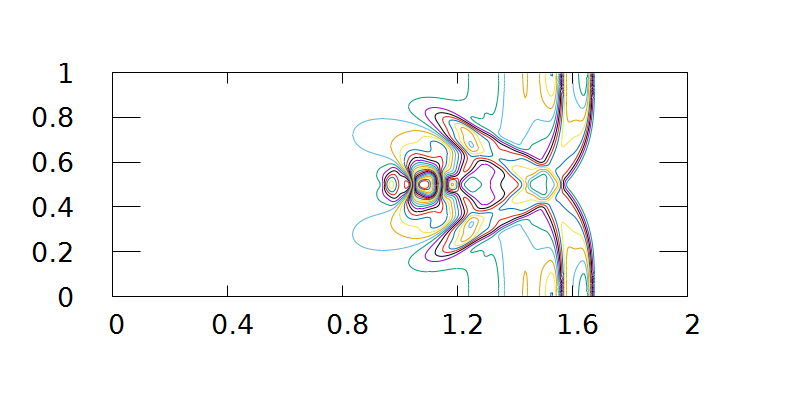}}
				}
\mbox{\subfigure[surface level at $t=0.6$]					{\includegraphics[width=7cm]{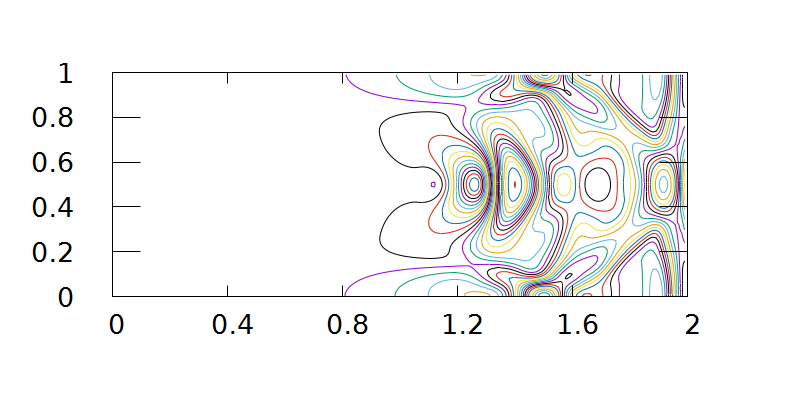}}\quad		
                     \subfigure[surface level at $t=0.6$]		
					{\includegraphics[width=7cm]{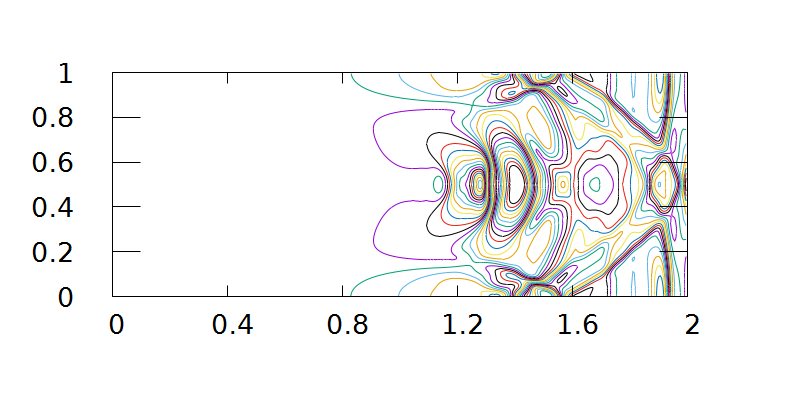}}
				}
				\caption{ Numerical solutions of the surface level $H=h+b$ for Example~\ref{exam8}. From top to bottom: at $t=0.12$ from 0.9993 to 1.0059; at  $t=0.24$ from 0.9980 to 1.0115; at  $t=0.36$ from 0.918 to 1.00872; $t=0.48$ from 0.9912 to 1.0044; $t=0.6$ from 0.9961 to 1.00432. 30 contour lines are used. Left: $200\times100$ uniform mesh. Right: $400\times200$ uniform mesh. }
				\label{Fg_eg8_1}
			\end{center}
		\end{figure}

}
\end{exa}

\begin{exa}{\em
		\label{exam7}
		({\bf{Traveling vortex}}) Now we consider a traveling vortex in the two dimensional case \cite{bispen2014imex}. The computational domain is $[0,2]\times[0,1]$, and the initial conditions are given by
		\begin{subequations}
			\label{ex7}
		\begin{equation}
        \label{Ex25}
		H(x,y,0) = 110+
		\left\{
		\begin{aligned}
		&\left(\frac{\eps \Gamma}{\omega}\right)^2(k(\omega\Gamma_c) - k(\pi)), \quad & \text{if }  \omega\Gamma_c \le\pi;\\
		&0, \quad & \text{otherwise},
		\end{aligned}
		\right.
		\end{equation}		
		\begin{equation}
        \label{Ex26}
		u(x,y,0) = 2+
		\left\{
		\begin{aligned}
		&\Gamma(1+\cos(\omega\Gamma_c))(0.5-y), \quad & \text{if }  \omega\Gamma_c \le\pi;\\
		&0, \quad & \text{otherwise},
		\end{aligned}
		\right.
		\end{equation}		
		\begin{equation}
        \label{Ex28}
		v(x,y,0) =
		\left\{
		\begin{aligned}
		&\Gamma(1+\cos(\omega\Gamma_c))(x-0.5), \quad & \text{if }  \omega\Gamma_c \le\pi;\\
		&0, \quad & \text{otherwise},
		\end{aligned}
		\right.
		\end{equation}
		\end{subequations}
	    where
        \begin{equation}
        \Gamma_c = \sqrt{(x-0.5)^2+(y-0.5)^2}, \quad \quad \Gamma = 8, \quad \omega = 4\pi,
        \end{equation}
         and
        \begin{equation}
        k(\xi) = 2\cos(\xi) + 2\xi\sin(\xi) + \frac{1}{8}\cos(2\xi) + \frac{\xi}{4}\sin(2\xi)+\frac{3}{4}\xi ^2.
        \end{equation}
		The center of the vortex is initially located at $(0.5, 0.5)$, and then propagates with a
		speed $u_{ref}=2$ along the horizontal direction. Periodic boundary conditions are used.
        For a flat bottom, the vortex could be referred as traveling only along the x-direction, where the exact solutions are given as follows \cite{ricchiuto2009stabilized}
		\begin{equation}
        \label{Ex7_1}
         H(x,y,t) = H(u-2t,y,0), \quad u(x,y,t) = u(x-2t,y,0), \quad v(x,y,t)= v(x-2t,y,0).
        \end{equation}	
        Note that the velocity can be decomposed as ${\bf u}={\bf u}_{ref}+{\bf u}'$, where $\bu_{ref}$ is the background traveling velocity and $\bu'$ is the rotating part which satisfies
       \[
       \nabla\cdot{\bf u}'=0, \quad
       ({\bf u}'\cdot\nabla){\bf u}'+\nabla H=0.
       \]
       Namely, the rotating part $\bu'$ is divergence free and balanced with $\nabla H$, so it performs as local self-rotating.
       {We show the numerical solutions on a mesh gird of $200\times100$ at the final time $T=1$ in Fig. ~{\ref{Fig7_1}} and Fig. ~{\ref{Fig7_2}},  and comparing our results to those produced by the explicit scheme of ``WB-Xing'' on the same mesh, for three different Froude numbers $\eps=1, 0.05, 0.01$. The perturbation of the water surface level $H$ from a constant level $110$ is at the scale of $\eps^2$. We can see that for large Froude number $\eps=1$, both schemes capture the traveling wave well. However, as the Froude number becomes small, e.g. $\eps=0.05$, our AP scheme can still keep the good shape of the vortex, while the results from the explicit ``WB-Xing'' scheme have been greatly damped, due to large numerical viscosities which are inversely proportional to the Froude number $\eps$.
       For the case of $\eps=0.01$, our AP scheme still has good performance, while the wave has been totally damped out for the ``WB-Xing'' scheme, and numerical noises from the damped wave spreading up to the boundary now pollute the whole computational domain, which is also the case for smaller $\eps$'s.
        }

\begin{figure}[hbtp]
			\begin{center}
				\mbox{\subfigure[$\eps = 1$ ]
	{\includegraphics[width=8cm]{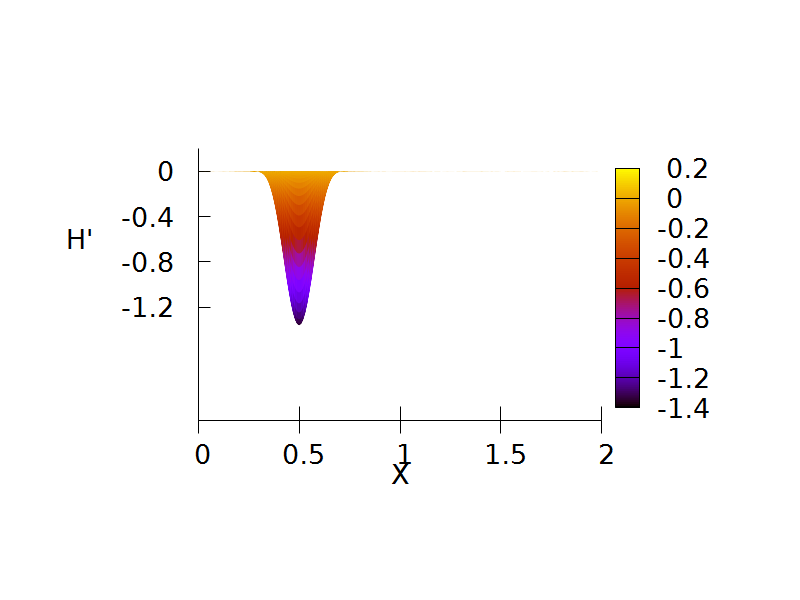}}\quad
	\subfigure[$\eps = 1$]
	{\includegraphics[width=8cm]{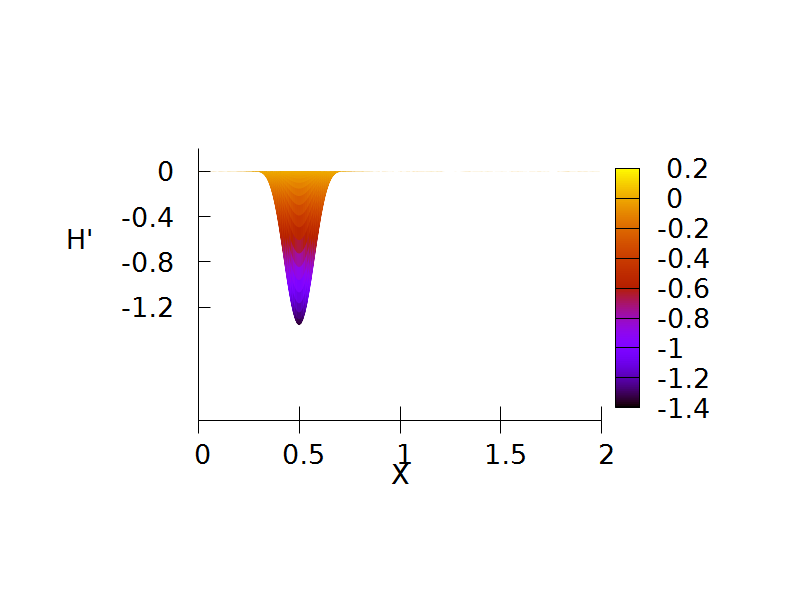}}}
\mbox{\subfigure[$\eps = 0.05$]				{\includegraphics[width=8cm]{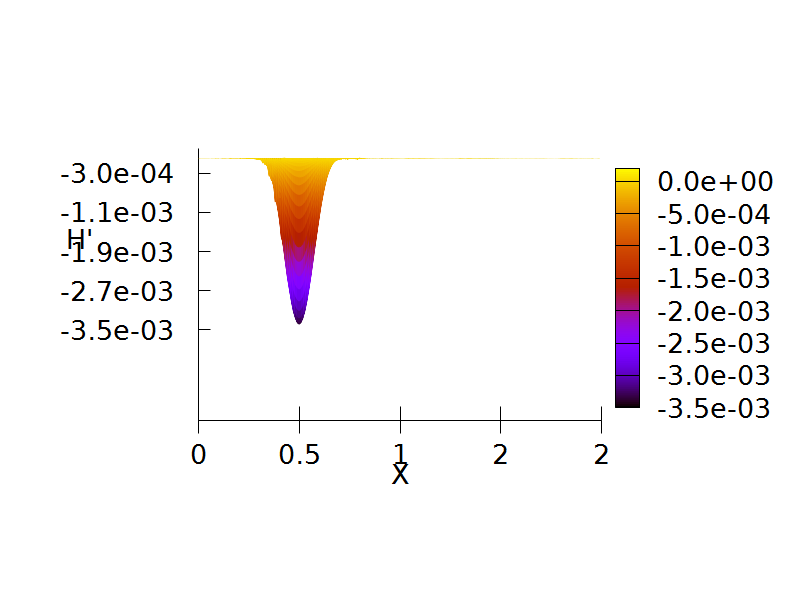}}
	\subfigure[$\eps = 0.05$]				{\includegraphics[width=8cm]{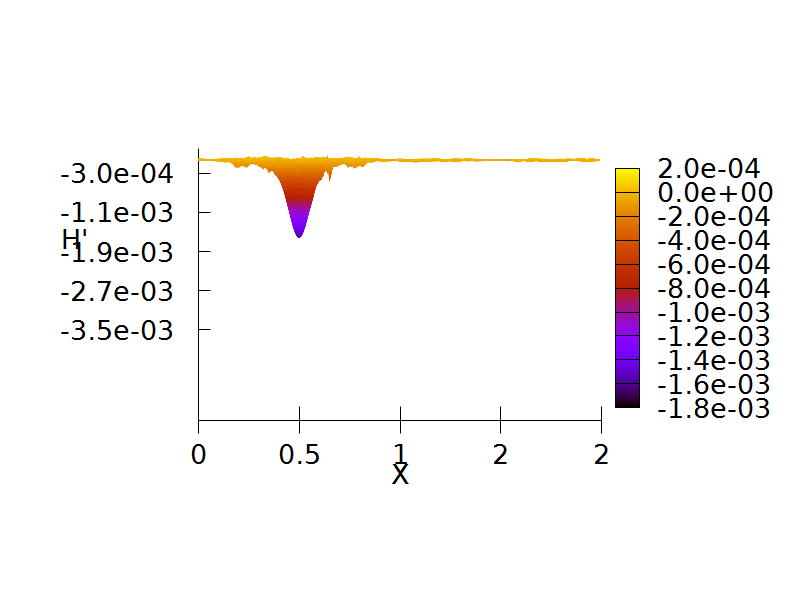}}}
\mbox{\subfigure[$\eps = 0.01$]				{\includegraphics[width=8cm]{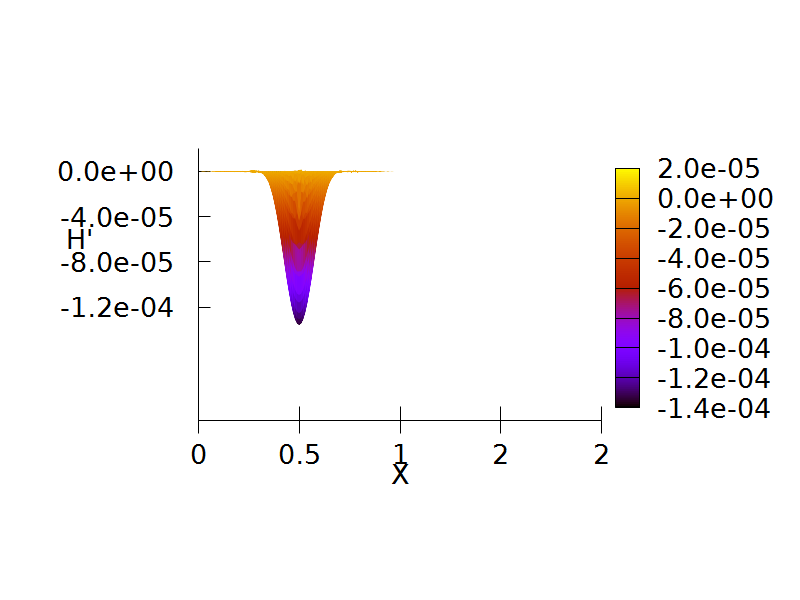}}
	\subfigure[$\eps = 0.01$]				{\includegraphics[width=8cm]{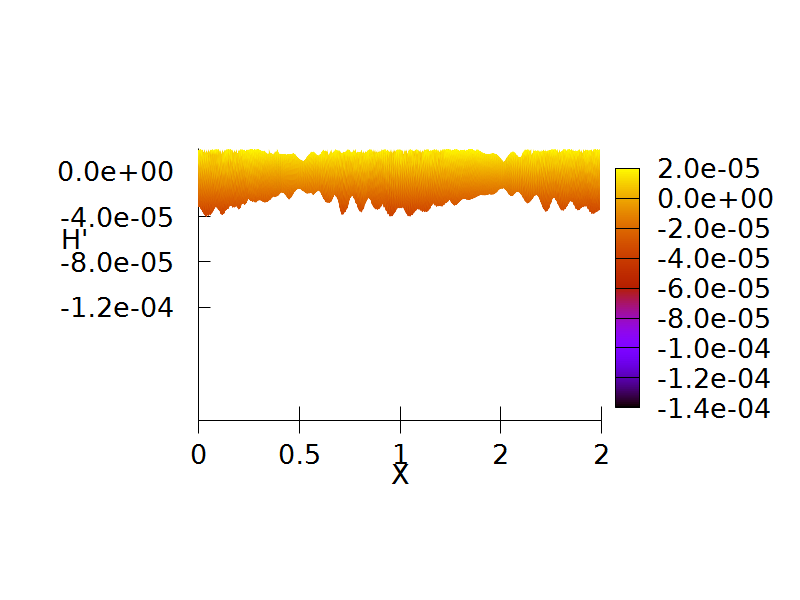}}}
\caption{ Example~\ref{exam7}. Numerical solutions about surface level for the traveling vortex at time $T=1$, on a mesh $200\times 100$. $H'=H-110$ is the deviation from the water surface level of $110$. From top to bottom $\eps=1, 0.05, 0.01$ respectively. Left: IMEX; Right: WB-Xing. }
\label{Fig7_1}
			\end{center}
		\end{figure}

\begin{figure}[hbtp]
	\begin{center}
		\mbox{\subfigure[$\eps = 1$ ]
			{\includegraphics[width=8cm]{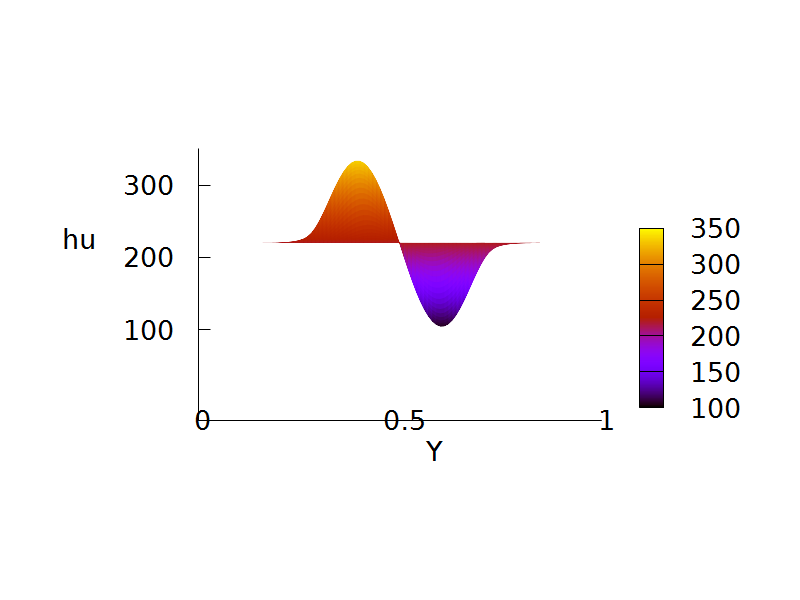}}\quad
			\subfigure[$\eps = 1$]
			{\includegraphics[width=8cm]{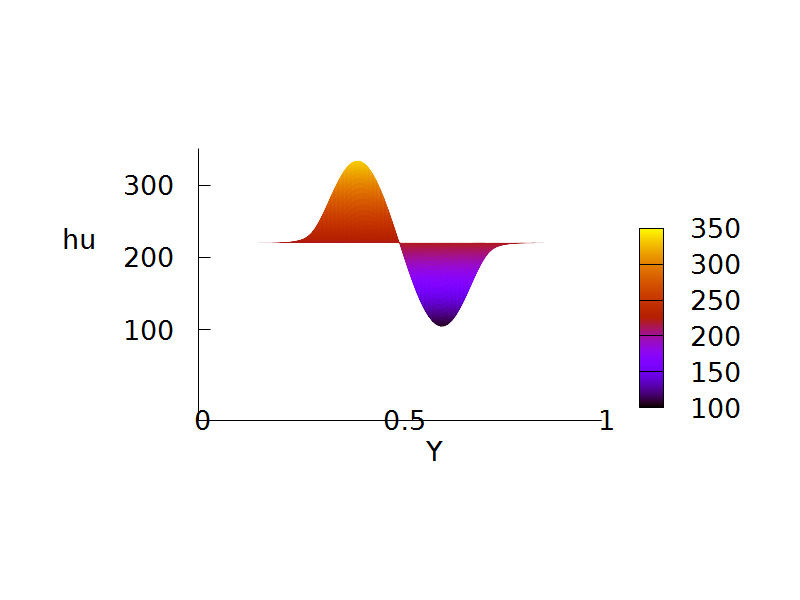}}}
		\mbox{\subfigure[$\eps = 0.05$]				{\includegraphics[width=8cm]{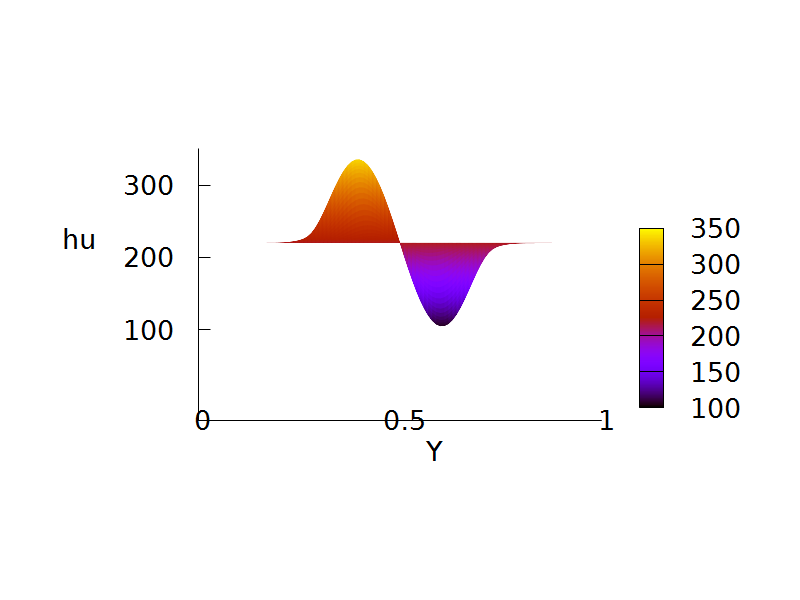}}
			\subfigure[$\eps = 0.05$]				{\includegraphics[width=8cm]{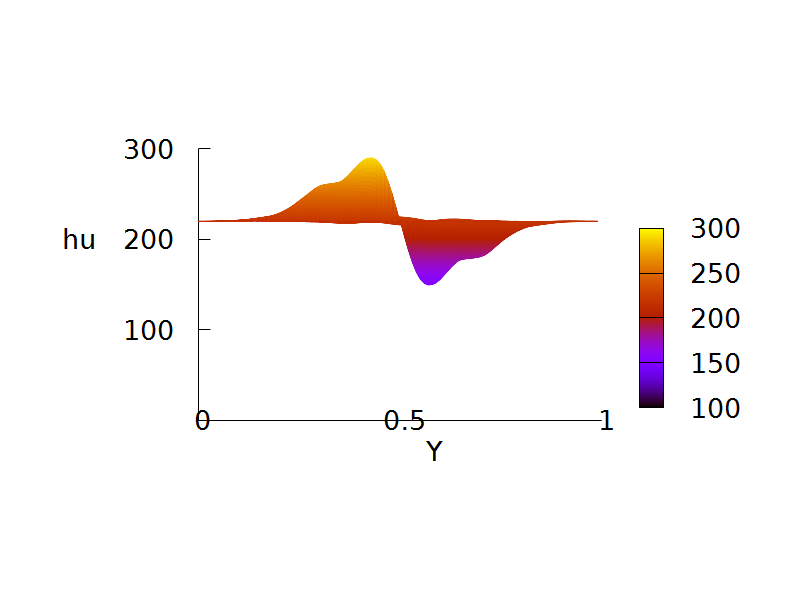}}}
		\mbox{\subfigure[$\eps = 0.01$]				{\includegraphics[width=8cm]{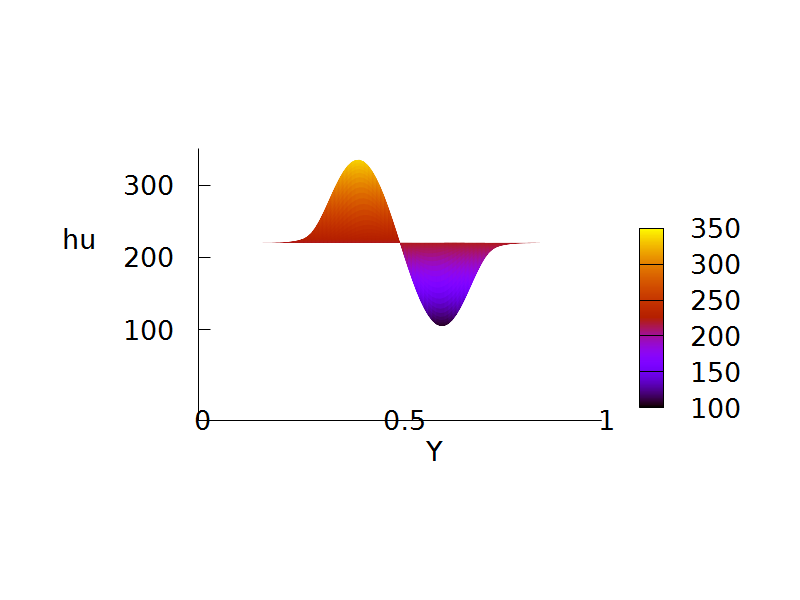}}
			\subfigure[$\eps = 0.01$]				{\includegraphics[width=8cm]{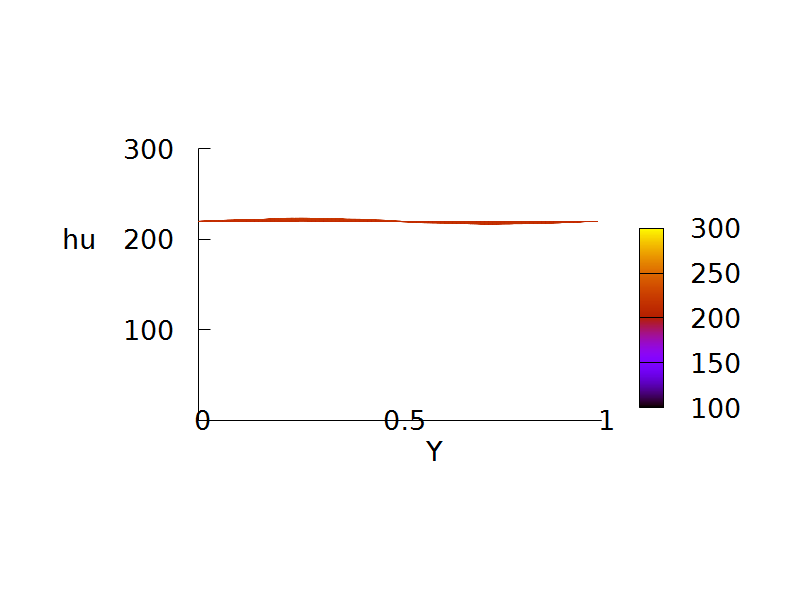}}}		%
		\caption{ Example~\ref{exam7}. Numerical solutions about momentum $hu$ for the traveling vortex at time $T=1$, on a mesh $200\times 100$. From top to bottom $\eps=1, 0.05, 0.01$ respectively. Left: IMEX; Right: WB-Xing.}
		\label{Fig7_2}
	\end{center}
\end{figure}
       Next we add a non-flat bottom which is variant in the $x$ direction,
       $$
       b(x,y) = e^{-5(x-1)^2},
       $$
and keep others the same as in \eqref{ex7}. In this case, the water surface level would be perturbed a little due to the non-flat bottom, but the vortex still travels almost the same. A similar example has been studied in \cite{bispen2014imex,liu2020well}. In Fig.~{\ref{Fig7_3}}, we show the numerical solutions at several different times $T=0, 0.3, 0.6, 1.0$ with $\eps$ chosen as $0.05$. The traveling vortex can also be well captured in this case. Similarly the solutions of the ``WB-Xing'' scheme have been damped.

       \begin{figure}[hbtp]
			\begin{center}
                 \mbox{\subfigure[$t = 0$ ]
				{\includegraphics[width=6.5cm]{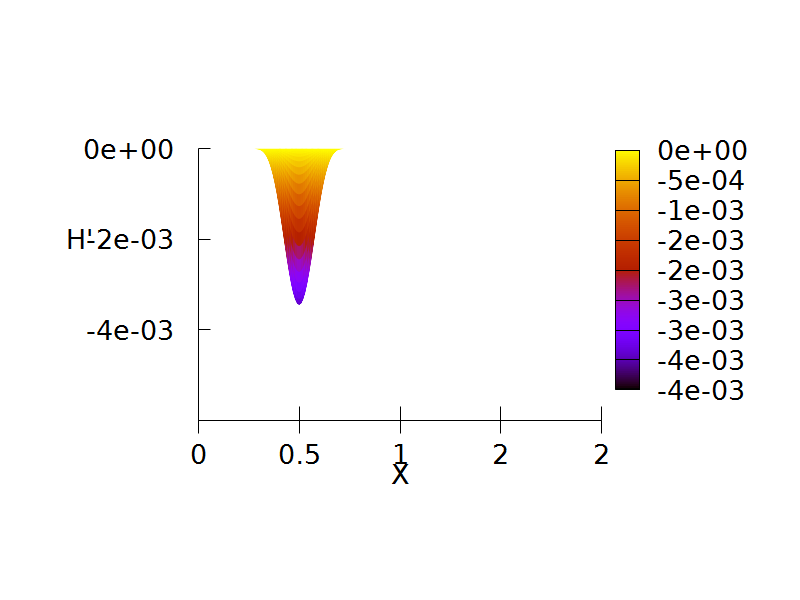}}\quad
                      \subfigure[$t = 0$]
				{\includegraphics[width=6.5cm]{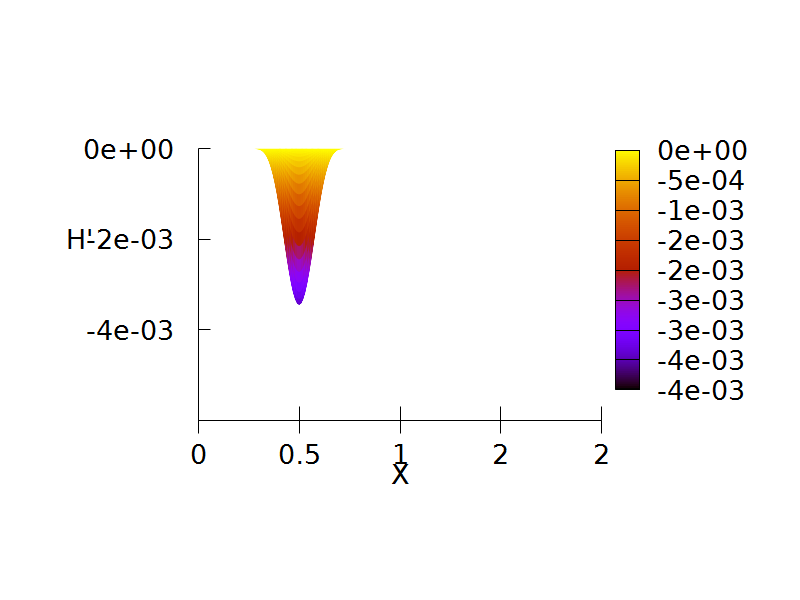}}}
		        \mbox{\subfigure[$t = 0.3$ ]
			{\includegraphics[width=6.5cm]{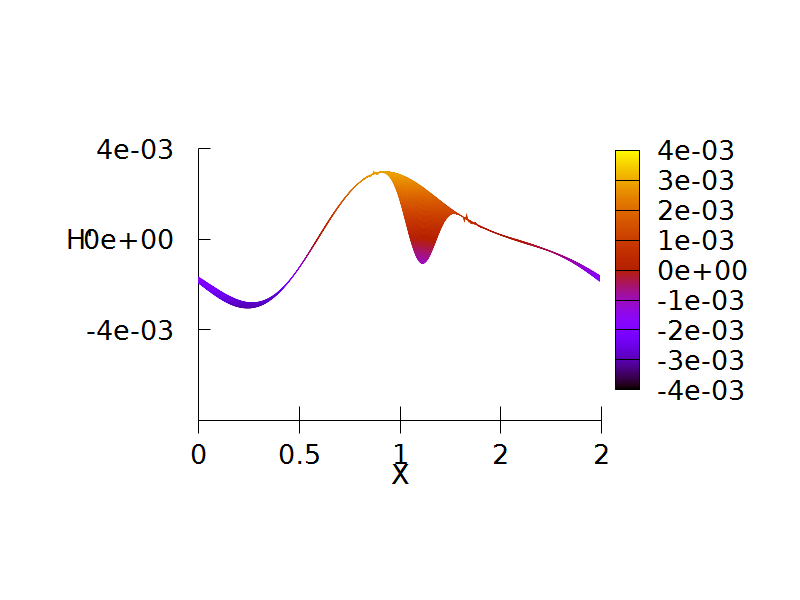}}\quad
                      \subfigure[$t = 0.3$]
				{\includegraphics[width=6.5cm]{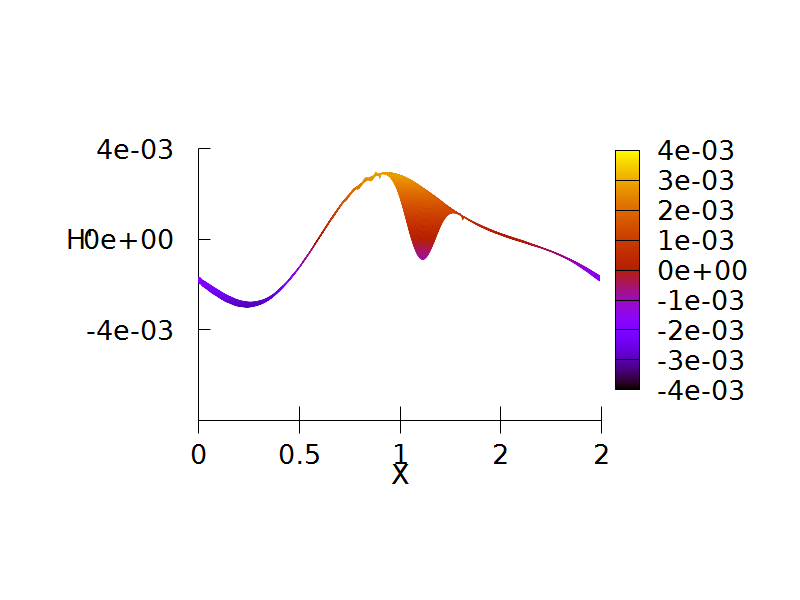}}}
                \mbox{\subfigure[$t = 0.6$ ]
			{\includegraphics[width=6.5cm]{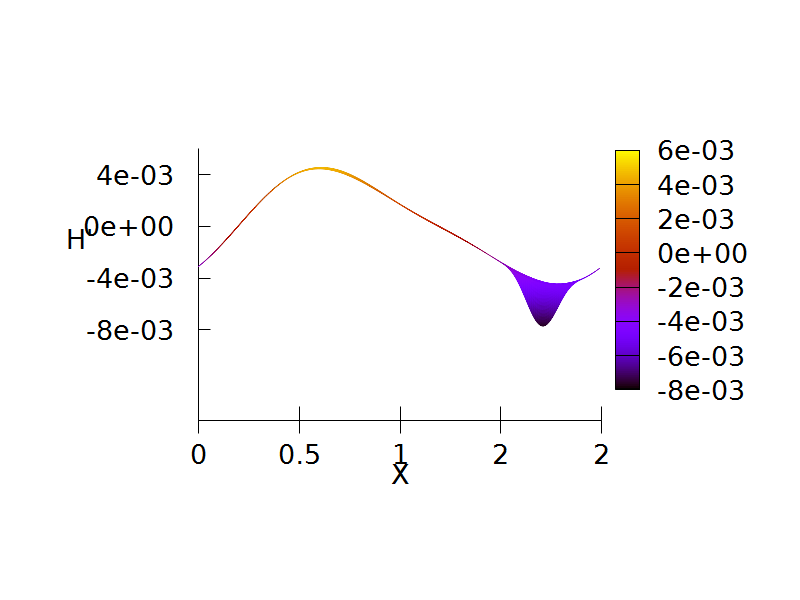}}\quad
                      \subfigure[$t = 0.6$]
				{\includegraphics[width=6.5cm]{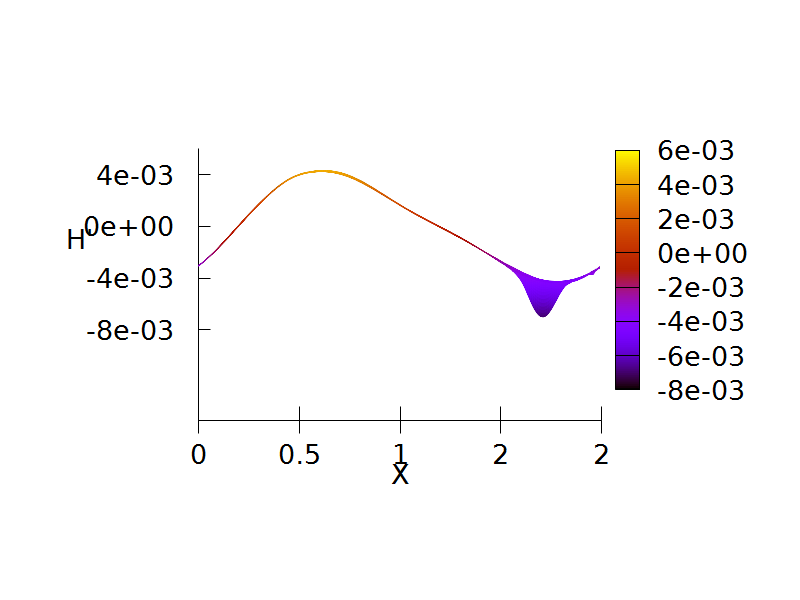}}}
                \mbox{\subfigure[$t = 1.0$ ]
			{\includegraphics[width=6.5cm]{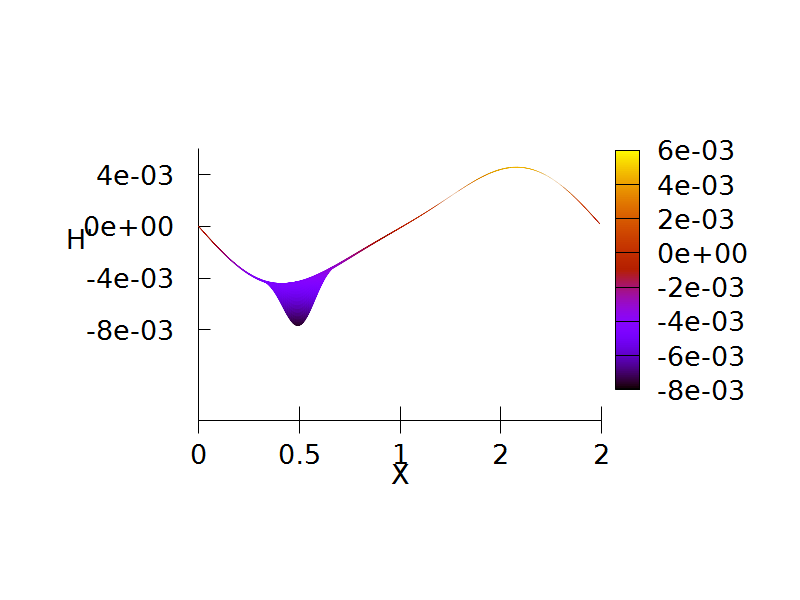}}\quad
                      \subfigure[$t = 1.0$]
				{\includegraphics[width=6.5cm]{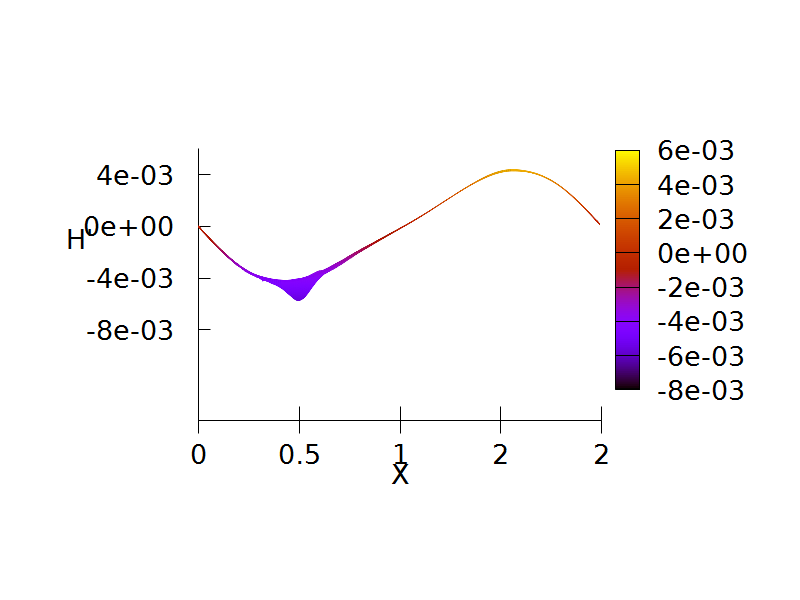}}}
          \caption{ {Example~\ref{exam7}. Numerical solutions of the traveling vortex with a nonflat bottom, on a mesh $200\times 100$. $H'=H-110$ is the deviation from the water surface level. From top to bottom $T=0, 0.3, 0.6, 1.0$ respectively, and $\eps=0.05$. Left: IMEX; Right: WB-Xing.}}
				\label{Fig7_3}
			\end{center}
		\end{figure}
{Finally, in Table~{\ref{T_eg8}} we compare the CPU cost for the two schemes with different $\eps$'s, for the cases considered above. We can find that, the CPU time of the IMEX scheme is less than the explicit ``WB-Xing'' scheme, especially in the low Froude regime, as the explicit method requires a much smaller time step for stability. Generally, the IMEX scheme would be much more efficient than the explicit one in the low Froude regime.}
\begin{table}[htbp]
  \caption{ Example \ref{exam7}. The CPU time (seconds) for two schemes with different Froude numbers, with flat and non-flat bottom topographies. }
  \begin{center}
  \begin{tabular}{c||c|c|c}\hline\hline
   Bottom topology &  $\eps$         &    IMEX    &    WB-Xing    \\ \hline\hline
   \multirow{3}{1cm}{$b=0$}
                   & 1             & 4071.7   & 5658.3
                   \\\cline{2-4}
                   & 0.05          & 5342.3   & 79818.8
                   \\\cline{2-4}
                   & 0.01          & 12444.2  & 378733.6  \\ \hline\hline
  \multirow{1}{1cm}{$b\neq 0$}
                   & 0.05          & 6634.5   & 81090.6  \\ \hline\hline
  \end{tabular}
  \end{center}
  \label{T_eg8}
\end{table}
}
\end{exa}

\section{Conclusion}
\label{sec5}
\setcounter{equation}{0}
\setcounter{figure}{0}
\setcounter{table}{0}

{In this paper, a high order semi-implicit asymptotic preserving scheme for the shallow water equations with a non-flat bottom topography is developed. The scheme is shown to be well-balanced, asymptotic preserving and asymptotically accurate. Numerical results in 1D and 2D have demonstrated the well-balanced property, the capability of capturing small perturbations of still water equilibrium, high order accuracy and asymptotic preserving for all ranges of Froude numbers. As compared to the explicit ``WB-Xing'' scheme, the semi-implicit AP scheme performs almost the same for large Froude numbers while capturing small perturbations well, and is in general much more efficient in the low Froude regime.}

%
%
%\section*{Acknowledgement}
%
%T. Xiong acknowledges support by TZ2016002, NSFC grant No. 11971025, NSF grant of Fujian Province No. 2019J06002

%%%%%%%%%%%%%%%%%%%%%%%%%%%%%%%%%%%%%%%%%%
%
%%%%%%%%%%%%%%%%%%%%%%%%%%%%%%%%%%%%%%%%%%

%\newpage

\bibliographystyle{abbrv}
\bibliography{refer}

%%%%%%%%%%%%%%%%%%%%%%%%%%%%%%%%%%
%
%%%%%%%%%%%%%%%%%%%%%%%%%%%%%%%%%%

\end{document}